\documentclass[11pt, twoside, leqno]{article}

\usepackage{amssymb}
\usepackage{amsmath}
\usepackage{amsthm}
\usepackage{amsfonts}
\usepackage{mathrsfs}
\usepackage{color}

\allowdisplaybreaks

\def\rr{{\mathbb R}}
\def\rn{{\mathbb{R}^n}}
\def\urn{\mathbb{R}_+^{n+1}}
\def\zz{{\mathbb Z}}
\def\cc{{\mathbb C}}

\def\nn{{\mathbb N}}
\def\hh{{\mathbb H}}

\def\ca{{\mathcal A}}
\def\cb{{\mathcal B}}

\def\cd{{\mathcal D}}

\def\cf{{\mathcal F}}
\def\cg{{\mathcal G}}

\def\cm{{\mathcal M}}
\def\cn{{\mathcal N}}
\def\cp{{\mathcal P}}

\def\cs{{\mathcal S}}
\def\ct{{\mathcal T}}

\def\fz{\infty }
\def\az{\alpha}

\def\ez{\epsilon}
\def\gz{{\gamma}}

\def\lz{\lambda}

\def\vz{\varphi}

\def\lf{\left}
\def\r{\right}

\def\hs{\hspace{0.25cm}}
\def\ls{\lesssim}

\def\noz{\nonumber}
\def\wz{\widetilde}
\def\wh{\widehat}

\def\st{\subset}
\def\com{\complement}

\def\loc{{\mathop\mathrm{\,loc\,}}}
\def\supp{\mathop\mathrm{\,supp\,}}

\def\rad{\mathop\mathrm{\,rad\,}}

\def\vlp{{L^{p(\cdot)}(\rn)}}
\def\vhs{H^{p(\cdot)}(\rn)}

\def\hlp{H_L^{p(\cdot)}(\rn)}

\newtheorem{theorem}{Theorem}[section]
\newtheorem{lemma}[theorem]{Lemma}
\newtheorem{corollary}[theorem]{Corollary}
\newtheorem{proposition}[theorem]{Proposition}

\theoremstyle{definition}

\newtheorem{remark}[theorem]{Remark}
\newtheorem{definition}[theorem]{Definition}
\newtheorem{assumption}[theorem]{Assumption}

\renewcommand{\appendix}{\par
   \setcounter{section}{0}%
   \setcounter{subsection}{0}%
   \setcounter{subsubsection}{0}%
   \gdef\thesection{\@Alph\c@section}%
   \gdef\thesubsection{\@Alph\c@section.\@arabic\c@subsection}%
   \gdef\theHsection{\@Alph\c@section.}%
   \gdef\theHsubsection{\@Alph\c@section.\@arabic\c@subsection}%
   \csname appendixmore\endcsname
 }

\numberwithin{equation}{section}

\begin{document}

\arraycolsep=1pt

\title{\bf\Large Maximal Function Characterizations of
Variable Hardy Spaces Associated with Non-negative Self-adjoint Operators
Satisfying Gaussian Estimates
\footnotetext{\hspace{-0.35cm} 2010 {\it
Mathematics Subject Classification}. Primary 42B25;
Secondary 42B30, 42B35, 35K08.
\endgraf {\it Key words and phrases.} non-negative self-adjoint operator,
Gaussian upper bound estimate, variable exponent Hardy space, atom,  non-tangential maximal function,
radial maximal function.
\endgraf This project is supported by the National
Natural Science Foundation of China
(Grant Nos.~11571039 and 11361020) and
the Specialized Research Fund for the Doctoral Program of Higher Education
of China (Grant No. 20120003110003). }}
\author{Ciqiang Zhuo and Dachun Yang\,\footnote{Corresponding author}
}
\date{}
\maketitle

\vspace{-0.8cm}

\begin{center}
\begin{minipage}{13cm}
{\small {\textbf{ Abstract}}\quad
Let $p(\cdot):\ \mathbb R^n\to(0,1]$ be a variable exponent function
satisfying the globally $\log$-H\"older continuous condition and $L$ a non-negative
self-adjoint operator on $L^2(\mathbb R^n)$ whose heat kernels
satisfying the Gaussian upper bound estimates.
Let $H_L^{p(\cdot)}(\mathbb R^n)$ be the variable exponent Hardy space
defined via the Lusin area function associated with the heat kernels $\{e^{-t^2L}\}_{t\in (0,\infty)}$.
In this article, the authors first establish the atomic characterization of $H_L^{p(\cdot)}(\mathbb R^n)$;
using this, the authors then obtain its non-tangential maximal
function characterization which, when $p(\cdot)$ is a constant in $(0,1]$, coincides
with a recent result by Song and Yan [Adv. Math. 287 (2016), 463-484]
and further induces the radial maximal function characterization of $H_L^{p(\cdot)}(\mathbb R^n)$
under an additional assumption that the heat kernels of $L$ have the H\"older regularity.}
\end{minipage}
\end{center}

\section{Introduction\label{s1}}
\hskip\parindent
The main purpose of this article is to establish the non-tangential or radial
maximal function characterizations of the Hardy space $H_L^{p(\cdot)}(\rn)$
introduced in \cite{yz15}.
Recall that the theory of classical Hardy spaces on the Euclidean space $\rn$ was introduced and
developed in the 1960s and 1970s. Precisely, the real-variable theory of Hardy spaces
on $\rn$ was initiated by Stein and Weiss \cite{sw60} and then systematically
developed by Fefferman and Stein \cite{fs72}, which
has played an important role in modern harmonic analysis
and been widely used in partial differential equations
(see, for example, \cite{CoWe77,fs72,st93}). As was well known,
the classical Hardy space is intimately connected with the Laplace operator
$\Delta:=-\sum_{i=1}^n\partial_{x_i}^2$ on $\rn$. Indeed, for $p\in(0,1]$,
the Hardy space $H^p(\rn)$ consists of all $f\in\cs'(\rn)$ (the set of all
\emph{tempered distributions}) such that the area integral function
$$S(f)(\cdot):=\lf\{\int_0^\fz\int_{|y-\cdot|<t}\lf|t^2\Delta e^{-t^2\Delta}(f)(y)\r|^2
\,\frac{dydt}{t^{n+1}}\r\}^{\frac12}$$
belongs to $L^p(\rn)$. Moreover, for $p\in(0,1]$, the Hardy space $H^p(\rn)$ involves several
different equivalent characterizations, for example, if $f\in \cs'(\rn)$, then
\begin{align*}
f\in H^p(\rn)
&\Longleftrightarrow\sup_{t\in(0,\fz)}\lf|e^{-t^2\Delta}(f)\r|\in L^p(\rn)\\
&\Longleftrightarrow \sup_{t\in(0,\fz),\,|y-\cdot|<t}
\lf|e^{-t^2\Delta}(f)(y)\r|\in L^p(\rn).
\end{align*}

Also, it is well known that the Hardy space $H^p(\rn)$, with $p\in(0,1]$, is a suitable
substitute of the Lebesgue space $L^p(\rn)$, for example, the classical Riesz transform
is bounded on $H^p(\rn)$, but not on $L^p(\rn)$ when $p\in(0,1]$. However, in many situations,
the standard theory of Hardy spaces is not applicable,
for example, the Riesz transform $\nabla L^{-1/2}$ may not be bounded from the Hardy
space $H^1(\rn)$ to $L^1(\rn)$ when $L$ is a second-order divergence form elliptic
operator with complex bounded measurable coefficients (see \cite{hm09}).
Motivated by this, the topic for developing a real-variable theory of Hardy spaces that are adapted
to different differential operators has inspired great interests in the last decade
and has become a very active research topic in harmonic analysis (see, for example,
\cite{admun,bckyy13a,dy05,dy05cpam,dz02,hlmmy,hm09,jy10,jyz09,yan08,yys4,yz15}).

Particularly, let $L$ be a linear operator on $L^2(\mathbb R^n)$ and
generate an analytic semigroup $\{e^{-tL}\}_{t>0}$ with heat kernels having
pointwise upper bounds.
Then, by using the Lusin area function associated with these heat kernels,
Auscher, Duong and McIntosh \cite{admun} initially studied
the Hardy space $H_L^1(\rn)$ associated with the operator $L$.
Based on this, Duong and Yan \cite{dy05,dy05cpam}
introduced the BMO-type space BMO$_L(\rn)$ associated with $L$
and proved that the dual space of $H_L^1(\rn)$ is just BMO$_{L^\ast}(\rn)$,
where $L^\ast$ denotes the \emph{adjoint operator} of $L$ in $L^2(\rn)$.
Later, Yan \cite{yan08} further generalized these results to the Hardy spaces
$H_L^p(\rn)$ with $p$ close to, but less than, 1 and, more generally, the
Orlicz-Hardy space associated with such operator was investigated by Jiang et al.
\cite{jyz09}. Very recently, under the assumption that
$L$ is a non-negative self-adjoint operator whose heat kernels satisfying Gaussian upper bound estimates,
Song and Yan \cite{sy15} established a characterization of Hardy spaces $H_L^p(\rn)$
via the non-tangential maximal function associated with the heat semigroup of $L$ based on a
subtle modification of technique due to Calder\'on \cite{c77},
which was further generalized into the Musielak-Orlicz-Hardy space in \cite{YYs15}.

Another research direction of generalized Hardy spaces
is the variable exponent Hardy space, which also extends the variable Lebesgue
space. Recall that the variable Lebesgue space $L^{p(\cdot)}(\rn)$, with a variable exponent
$p(\cdot):\ \rn\to(0,\fz)$, consists of all measurable functions $f$ such
that $\int_\rn |f(x)|^{p(x)}\,dx<\fz$. The study of
variable Lebesgue spaces can be traced back to
Birnbaum-Orlicz \cite{bo31} and Orlicz \cite{ol32}, but the modern development
started with the article \cite{kr91} of Kov\'a\v{c}ik
and R\'akosn\'{\i}k as well as \cite{cruz03} of
Cruz-Uribe and \cite{din04} of Diening, and nowadays have been widely used in
harmonic analysis (see, for example, \cite{cfbook,{dhr11}}).
Moreover, variable function spaces also have interesting applications in
fluid dynamics \cite{am02,{rm00}}, image processing \cite{clr06},
partial differential equations and variational calculus
\cite{am05,hhl08,su09}.
Recall that the variable exponent Hardy space $\vhs$ was introduced
by Nakai and Sawano \cite{ns12} and, independently, by Cruz-Uribe and Wang \cite{cw14}
with some weaker assumptions on $p(\cdot)$ than those used in \cite{ns12}, which
was further investigated by Sawano \cite{Sa13}, Zhuo et al. \cite{zyl14} and Yang et al. \cite{yzn16}.

Let $p(\cdot):\ \mathbb R^n\to(0,1]$ be a variable exponent function
satisfying the globally $\log$-H\"older continuous condition.
Very recently, the authors \cite{yz15} introduced the Hardy space
$H_L^{p(\cdot)}(\rn)$ via the Lusin area function associated with
a linear operator $L$ on $L^2(\mathbb R^n)$ whose heat kernels having pointwise
upper bound, and obtained their molecular characterizations.
In this article, we aim at establishing equivalent characterizations
of $H_L^{p(\cdot)}(\rn)$, under the additional assumption that $L$ is a non-negative
self-adjoint operator, in terms of maximal functions, including (grand)
non-tangential maximal functions and (grand) radial maximal function.
To this end, we first introduce the space $H_{L,{\rm at},M}^{p(\cdot),q}(\rn)$,
the variable atomic Hardy space associated with the operator $L$
(see Definition \ref{d-atom} below),
and then prove that $H_L^{p(\cdot)}(\rn)$ and $H_{L,{\rm at},M}^{p(\cdot),q}(\rn)$
coincide with equivalent quasi-norms (see Theorem \ref{t-atom-c} below).
Based on the results from Song and Yan \cite{sy15},
we characterize $H_{L,{\rm at},M}^{p(\cdot),q}(\rn)$ via (grand) non-tangential maximal
functions in Theorem \ref{t-max-c} below, from which, we further deduce
the (grand) radial maximal function characterizations of
$H_{L,{\rm at},M}^{p(\cdot),q}(\rn)$
in Theorem \ref{t-max-rad} below under an additional H\"older continuous assumption
on the heat kernels of $L$ (see \eqref{11-25} below). Using Theorems \ref{t-atom-c},
\ref{t-max-c} and \ref{t-max-rad}, we also obtain the corresponding characterizations
of $H_L^{p(\cdot)}(\rn)$, respectively, in terms of atoms, the (grand) non-tangential maximal functions
and the (grand) radial maximal functions (see Corollary \ref{c-max} below).

To state the results of this article, we begin with some notation and notions.
A measurable function $p(\cdot):\ \rn\to(0,\fz)$ is called a
\emph{variable exponent}. For any variable exponent $p(\cdot)$, let
\begin{equation}\label{2.1x}
p_-:=\mathop{\rm ess\,inf}\limits_{x\in \rn}p(x)
\quad {\rm and}\quad
p_+:=\mathop{\rm ess\,sup}\limits_{x\in \rn}p(x).
\end{equation}
Denote by $\cp(\rn)$ the \emph{collection of all variable exponents}
$p(\cdot)$ \emph{satisfying} $0<p_-\le p_+<\fz$.

For a given variable exponent $p(\cdot)\in\cp(\rn)$,
the \emph{modular} $\varrho_{p(\cdot)}$, associated with $p(\cdot)$, is defined by setting
$\varrho_{p(\cdot)}(f):=\int_\rn|f(x)|^{p(x)}\,dx$ for any measurable function $f$ and the
\emph{Luxemburg (quasi-)norm} of $f$ is given by
\begin{equation*}
\|f\|_{\vlp}:=\inf\lf\{\lz\in(0,\fz):\ \varrho_{p(\cdot)}(f/\lz)\le1\r\}.
\end{equation*}
Then the \emph{variable exponent Lebesgue space} $\vlp$ is defined to be the
set of all measurable functions $f$ such that $\varrho_{p(\cdot)}(f)<\fz$,
equipped with the quasi-norm $\|f\|_{\vlp}$. For more properties on the variable
exponent Lebesgue spaces, we refer the reader to \cite{cfbook,dhr11}.

\begin{remark}\label{r-vlp}
Let $p(\cdot)\in\cp(\rn)$. Then, for all $f,\ g\in\vlp$,
$$\|f+g\|_{\vlp}^{\underline{p}}
\le \|f\|_{\vlp}^{\underline{p}}+\|g\|_{\vlp}^{\underline{p}},$$
where $\underline{p}:=\min\{1,p_-\}$,
and, for all $\lz\in\cc$, $\|\lz f\|_{\vlp}=|\lz|\|f\|_{\vlp}$.
In particular, when $p_-\in[1,\fz)$, $\vlp$ is a Banach space
(see \cite[Theorem 3.2.7]{dhr11}).
\end{remark}

In the present article, wa always assume that the variable exponent
$p(\cdot)$ satisfies the \emph{globally {\rm log}-H\"older continuous condition}.
Recall that a measurable function $p(\cdot)$ is said to
satisfy the globally {\rm log}-H\"older continuous condition, denoted by
$p(\cdot)\in  C^{\log}(\rn)$,
if there exists a positive constant $C_{\log}(p)$ such that, for all $x,\ y\in\rn$,
\begin{equation*}
|p(x)-p(y)|\le \frac{C_{\log}(p)}{\log(e+1/|x-y|)},
\end{equation*}
and there exist a positive constant
$C_\fz$ and a constant $p_\fz\in\rr$
such that, for all $x\in\rn$,
\begin{equation*}
|p(x)-p_\fz|\le \frac{C_\fz}{\log(e+|x|)}.
\end{equation*}

In what follows, for any $r\in(0,\fz)$ and measurable set $E\st \rn$,
denote by $L^r(E)$ the set of all measurable functions $f$ such that
$\|f\|_{L^r(E)}:=\{\int_E|f(x)|^r\,dx\}^{1/r}<\fz.$

In this article, unless otherwise stated, we always assume that
$L$ is a densely defined linear operator on $L^2(\rn)$ and satisfies
the following assumptions:
\begin{assumption}\label{a-a1}
$L$ is non-negative and self-adjoint;
\end{assumption}
\begin{assumption}\label{a-a2}
The kernels of the semigroup $\{e^{-tL}\}_{t>0}$, denoted by
$\{K_t\}_{t>0}$, are measurable functions on $\rn\times\rn$ and satisfy the Gaussian
upper bound estimates, namely, there exist positive constants $C$ and $c$ such that,
for all $t\in(0,\fz)$ and $x,\,y\in\rn$,
\begin{equation*}
|K_t(x,y)|\le \frac C{t^{n/2}}\exp\lf\{-\frac{|x-y|^2}{ct}\r\}.
\end{equation*}
\end{assumption}

\begin{remark}
(i) One of the typical example of operators $L$ satisfying both Assumptions \eqref{a-a1}
and \ref{a-a2} is the Schr\"odinger operator $L:=-\Delta+V$ with $0\le V\in L_{\loc}^1(\rn)$.

(ii) If $\{e^{-tL}\}_{t>0}$ is a bounded analytic semigroup on $L^2(\rn)$ whose
kernels $\{K_t\}_{t>0}$ satisfy Assumptions \ref{a-a1} and \ref{a-a2}, then, for
any $j\in\nn:=\{1,2,\dots\}$, there exists a positive constant $C$ such that, for all $t\in(0,\fz)$
and almost every $x,\ y\in\rn$,
\begin{equation}\label{partial1}
\lf|t^j\frac{\partial^jK_t(x,y)}{\partial t^j}\r|\le \frac C{t^{n/2}}
\exp\lf\{-\frac{|x-y|^2}{ct}\r\};
\end{equation}
see, for example, \cite[p.\,4386]{yan08}.
\end{remark}
For all functions $f\in L^2(\rn)$, define the \emph{Lusin area function} $S_L(f)$
associated with the operator $L$ by setting, for all $x\in\rn$,
\begin{equation*}
S_L(f)(x):=\lf\{\int_{\Gamma(x)}\lf|t^2Le^{-t^2L}(f)(y)\r|^2\,
\frac{dy\,dt}{t^{n+1}}\r\}^{1/2},
\end{equation*}
here and hereafter, for all $x\in\rn$,
\begin{equation}\label{1.2x}
\Gamma(x):=\{(y,t)\in\rn\times(0,\fz):\ |y-x|<t\}.
\end{equation}
In \cite{admun}, Auscher et al. proved that, for any $p\in(1,\fz)$,
there exists a positive
constant $C$ such that, for all $f\in L^p(\rn)$,
\begin{equation}\label{SL-bounded}
C^{-1}\|f\|_{L^p(\rn)}\le \|S_L(f)\|_{L^p(\rn)}\le C\|f\|_{L^p(\rn)};
\end{equation}
see also Duong and McIntosh \cite{dm99} and Yan \cite{yan04}.

We now recall the definition of the variable exponent Hardy space associated with operator,
which was first studied in \cite{yz15}.

\begin{definition}\label{d-hardys}
Let $p(\cdot)\in C^{\log}(\rn)$ with $p_+\in(0,1]$ and
$L$ be an operator satisfying Assumptions \ref{a-a1} and \ref{a-a2}.
A function $f\in L^2(\rn)$ is said to be in $\hh_L^{p(\cdot)}(\rn)$ if
$S_L(f)\in L^{p(\cdot)}(\rn)$; moreover, define
$$\|f\|_{H_L^{p(\cdot)}(\rn)}
:=\|S_L(f)\|_{L^{p(\cdot)}(\rn)}
:=\inf\lf\{\lz\in(0,\fz):\ \int_\rn\lf[\frac{S_L(f)(x)}{\lz}\r]^{p(x)}\,dx\le1\r\}.$$
Then the \emph{variable Hardy space associated with operator} $L$,
denoted by $H_L^{p(\cdot)}(\rn)$,
is defined to be the completion of $\hh_L^{p(\cdot)}(\rn)$ in the quasi-norm
$\|\cdot\|_{\hlp}$.
\end{definition}

Next we introduce the notions of the $(p(\cdot),q,M)_L$-atom and the atomic variable exponent Hardy
space $H_{L,{\rm at},M}^{p(\cdot),q}(\rn)$.

\begin{definition}\label{d-atom}
Let $L$ and $p(\cdot)$ be as in Definition \ref{d-hardys}, $q\in(1,\fz]$ and $M\in\nn$.

(I) Let $\cd(L^M)$ be the domain of $L^M$ and $Q\st\rn$ a cube.
A function $\alpha\in L^q(\rn)$ is called
a \emph{$(p(\cdot),q,M)_L$-atom} associated with the cube $Q$ if there exists
a function $b\in\cd(L^M)$ such that
\begin{enumerate}
\item[(i)] $\alpha=L^Mb$ and, for all $j\in\{0,\,1,\,\dots,\,M\}$,
$\supp(L^jb)\st Q$;

\item[(ii)] for all $j\in\{0,\,1,\,\dots,\,M\}$, $\|([\ell(Q)]^2L)^jb\|_{L^q(\rn)}
\le [\ell(Q)]^{2M}|Q|^{1/q}\|\chi_Q\|_{\vlp}^{-1}$, where $\ell(Q)$ denotes
the \emph{side length} of $Q$.
\end{enumerate}

(II) Let $f\in L^2(\rn)$. Then
\begin{equation}\label{atom-re}
f=\sum_{j\in\nn}\lz_j\alpha_j
\end{equation}
is called an atomic \emph{$(p(\cdot),q,M)_L$-representation of $f$} if
the sequences $\{\lz_j\}_{j\in\nn}\st\cc$ and $\{\alpha_j\}_{j\in\nn}$ are
$(p(\cdot),q,M)_L$-atoms associated with cubes $\{Q_j\}_{j\in\nn}\st\rn$
such that \eqref{atom-re} converges in $L^2(\rn)$ and
$$\int_\rn\lf\{\sum_{j\in\nn}\lf[\frac{|\lz_j|\chi_{Q_j}(x)}{\|\chi_{Q_j}\|_{\vlp}}
\r]^{p_-}\r\}^{p(x)/p_-}\,dx<\fz.$$

Let
$$\mathbb H_{L,{\rm at},M}^{p(\cdot),q}(\rn):=\lf\{f\in L^2(\rn):\
f\ \text{has an atomic}\ (p(\cdot),q,M)_L\text{-representation}\r\}$$
equipped with the quasi-norm $\|f\|_{{H}_{L,{\rm at},M}^{p(\cdot),q}(\rn)}$
given by
$$\inf\lf\{\cb(\{\lz_j\alpha_j\}_{j\in\nn}):\ \sum_{j\in\nn}\lz_j\alpha_j\
\text{is an atomic}\ (p(\cdot),q,M)_L \text{-representation of}\ f\r\},$$
where
\begin{equation*}
\cb(\{\lz_j\alpha_j\}_{j\in\nn})
:=\lf\|\lf\{\sum_{j\in\nn}\lf[\frac{|\lz_j|\chi_{Q_j}}{\|\chi_{Q_j}\|_{\vlp}}
\r]^{p_-}\r\}^{1/p_-}\r\|_{\vlp}
\end{equation*}
and the infimum is taken over all the atomic $(p(\cdot),q,M)_L$-representations
of $f$ as above.

The \emph{atomic variable exponent Hardy space} $H_{L,{\rm at},M}^{p(\cdot),q}(\rn)$
is then defined to be the completion of the set
$\mathbb H_{L,{\rm at},M}^{p(\cdot),q}(\rn)$ with respect to the quasi-norm
$\|\cdot\|_{H_{L,{\rm at},M}^{p(\cdot),q}(\rn)}$.
\end{definition}

\begin{remark}\label{r-atom}
It is easy to see that, for any $q\in(1,\fz)$ and $M\in\nn$,
$$H_{L,{\rm at},M}^{p(\cdot),\fz}(\rn)\st H_{L,{\rm at},M}^{p(\cdot),q}(\rn).$$
\end{remark}

The first main result of this article is stated as follows, which, in the case that
$p(\cdot)\equiv {\rm constant}\in(0,1]$, was established in \cite{dl13,hlmmy}
(see also \cite{jy11}).

\begin{theorem}\label{t-atom-c}
Let $p(\cdot)\in C^{\log}(\rn)$ with $p_+\in(0,1]$, $q\in(1,\fz]$,
$M\in(\frac n{2}[\frac1{p_-}-1],\fz)\cap\nn$ and $L$ be a linear operator on $L^2(\rn)$
satisfying Assumptions \ref{a-a1} and \ref{a-a2}.
Then $H_{L,{\rm at},M}^{p(\cdot),q}(\rn)$ and $H_L^{p(\cdot)}(\rn)$ coincide
with equivalent quasi-norms.
\end{theorem}

In this article, we use $\cs(\rn)$ to denote the space of all
Schwartz functions on $\rn$.

\begin{definition}\label{d-max-f}
(i) Let $\phi\in\cs(\rr)$ be an even function with $\phi(0)=1$.
For any $a\in(0,\fz)$ and $f\in L^2(\rn)$, the \emph{non-tangential maximal function}
of $f$ is defined by setting, for all $x\in\rn$,
\begin{equation*}
\phi_{L,\triangledown,a}^\ast(f)(x):=\sup_{t\in(0,\fz),\,|y-x|<at}
\lf|\phi(t\sqrt L)(f)(y)\r|.
\end{equation*}
A function $f\in L^2(\rn)$ is said to be in the set
$\hh_{L,\max}^{p(\cdot),\phi,a}(\rn)$ if $\phi_{L,\triangledown,a}^\ast(f)\in\vlp$;
moreover, define
$\|f\|_{H_{L,\max}^{p(\cdot),\phi,a}(\rn)}
:=\|\phi_{L,\triangledown,a}^\ast(f)\|_{\vlp}$.
Then the \emph{variable exponent Hardy space} $H_{L,\max}^{p(\cdot),\phi,a}(\rn)$
is defined to be the completion of $\hh_{L,\max}^{p(\cdot),\phi,a}(\rn)$ with respect
to the quasi-norm $\|\cdot\|_{H_{L,\max}^{p(\cdot),\phi,a}(\rn)}$.

Particularly, when $\phi(x):=e^{-x^2}$ for all $x\in\rn$,
we use $f_{L,\triangledown}^\ast$ to denote $\phi_{L,\triangledown,1}^\ast(f)$ and,
in this case, denote the space $H_{L,\max}^{p(\cdot),\phi,a}(\rn)$
simply by $H_{L,\max}^{p(\cdot)}(\rn)$.

(ii) For any $f\in L^2(\rn)$, define the \emph{grand non-tangential maximal
function} of $f$ by setting, for all $x\in\rn$,
$$\cg_{L,\triangledown}^\ast(f)(x)
:=\sup_{\phi\in \cf(\rr)}\phi_{L,\triangledown,1}^\ast(f)(x),$$
where $\cf(\rr)$ denotes the set of all even functions $\phi\in\cs(\rr)$
satisfying $\phi(0)\neq0$ and
$$
\sum_{k=0}^N\int_\rr(1+|x|)^N\lf|\frac{d^k\phi(x)}{dx^k}\r|^2\,dx
\le1
$$
with $N$ being a large enough number depending on $p(\cdot)$ and $n$.
Then the \emph{variable exponent Hardy space} $H_{L,\max}^{p(\cdot),\cf}(\rn)$
is defined in the same way as $H_{L,\max}^{p(\cdot),\phi,a}(\rn)$
but with $\phi_{L,\triangledown,a}^\ast(f)$ replaced by
$\cg_{L,\triangledown}^\ast(f)$.
\end{definition}

\begin{remark}\label{r-1.7x}
By Assumption \ref{a-a2}, we conclude that there exists a positive constant
$C$ such that, for any $f\in L^2(\rn)$ and $x\in\rn$,
$f_{L,\triangledown}^\ast(x)\le C\cm(f)(x)$.
Here and hereafter, $\cm$ denotes the \emph{Hardy-Littlewood maximal operator}, which is
defined by setting, for all locally integrable function $f$ and $x\in\rn$,
$$\cm(f)(x):=\sup_{B\ni x}\frac1{|B|}\int_B|f(y)|\,dy,$$
where the supremum is taken over all balls $B$ of $\rn$.
\end{remark}

The second main result of this article is presented as follows.

\begin{theorem}\label{t-max-c}
Let $p(\cdot)\in C^{\log}(\rn)$ with $p_+\in(0,1]$, $q\in(1,\fz]$,
$M\in(\frac n{2}[\frac1{p_-}-1],\fz)$ and
$L$ be an operator satisfying Assumptions \ref{a-a1} and \ref{a-a2}.
Then, for any $a\in(0,\fz)$ and $\phi$ as in Definition \ref{d-max-f},
the spaces $H_{L,{\rm at},M}^{p(\cdot),q}(\rn)$,
$H_{L,\max}^{p(\cdot),\cf}(\rn)$ and $H_{L,\max}^{p(\cdot),\phi,a}(\rn)$
coincide with equivalent quasi-norms.
\end{theorem}

\begin{remark}
When $p(\cdot)\equiv{\rm constant}\in(0,1]$, the conclusion of Theorem \ref{t-max-c} was
proved by Song and Yan in \cite[Theorem 1.4]{sy15}.
\end{remark}

\begin{definition}
(i) Let $\phi\in\cs(\rr)$ be an even function with $\phi(0)=1$.
For $f\in L^2(\rn)$ and $x\in\rn$, let
\begin{equation*}
\phi_{L,+}^\ast(f)(x):=\sup_{t\in(0,\fz)}\lf|\phi(t\sqrt L)(f)(x)\r|.
\end{equation*}
Particularly, when $\phi(x):=e^{-x^2}$ for all $x\in\rn$,
we use $f_{L,+}^\ast$ to denote $\phi_{L,+}^\ast(f)$.
The \emph{variable exponent Hardy space} $H_{L,\rad}^{p(\cdot)}(\rn)$
is defined in the same way as $H_{L,\max}^{p(\cdot),\phi,a}(\rn)$ but with
$\phi_{L,\triangledown,a}^\ast(f)$ replaced by
$f_{L,+}^\ast$.

(ii) For any $f\in L^2(\rn)$ and $x\in\rn$, let
$$\cg_{L,+}^\ast(f)(x):=\sup_{\phi\in \ca(\rr)}\phi_{L,+}^\ast(f)(x).$$
The \emph{variable exponent Hardy space} $H_{L,\rad}^{p(\cdot),\cf}(\rn)$ is defined
in the same way as $H_{L,\max}^{p(\cdot),\phi,a}(\rn)$ but with
$\phi_{L,\triangledown,a}^\ast(f)$ replaced by
$\cg_{L,+}^\ast(f)$.
\end{definition}

\begin{remark}\label{r-dense}
We point out that, for any $q\in(1,\fz]$ and $M\in\nn$,
the sets
\begin{center}
$H_{L,\rm{at},M}^{p(\cdot),q}(\rn)\cap L^2(\rn)$,
$H_{L,\max}^{p(\cdot),\phi,a}(\rn)\cap L^2(\rn)$,
$H_{L,\max}^{p(\cdot),\cf}(\rn)\cap L^2(\rn)$
\end{center}
and
\begin{center}
$H_{L,\rad}^{p(\cdot),\phi,a}(\rn)\cap L^2(\rn)$,
$H_{L,\rad}^{p(\cdot),\cf}(\rn)\cap L^2(\rn)$
\end{center}
are, respectively, dense in the spaces
$H_{L,\rm{at},M}^{p(\cdot),q}(\rn)$, $H_{L,\max}^{p(\cdot),\phi,a}(\rn)$,
$H_{L,\max}^{p(\cdot),\cf}(\rn)$, $H_{L,\rad}^{p(\cdot),\phi,a}(\rn)$
and $H_{L,\rad}^{p(\cdot),\cf}(\rn)$.
\end{remark}

By the definitions of $H_{L,\max}^{p(\cdot)}(\rn)$ and
$H_{L,\rad}^{p(\cdot)}(\rn)$, we
easily know that the continuous inclusion
$H_{L,\max}^{p(\cdot)}(\rn)\st H_{L,\rad}^{p(\cdot)}(\rn)$ holds true.
It is a natural question whether or not the continuous inclusion
\begin{equation}\label{11-25x}
H_{L,\rad}^{p(\cdot)}(\rn)\st H_{L,\max}^{p(\cdot)}(\rn)
\end{equation}
holds true.
We remark that, in the case of $p(\cdot)\equiv{\rm constant}\in(0,1]$,
\eqref{11-25x} has been proved in \cite[Theorem 1.9]{YYs15}
under the following additional Assumption \ref{a-a3} on the operator $L$,
which gives an affirmative answer to the open question stated in
\cite[Remark 3.4]{sy15}.

\begin{assumption}\label{a-a3}
There exist positive constants $C$ and $\mu\in(0,1]$ such that,
for all $t\in(0,\fz)$ and $x,\ y_1,\ y_2\in\rn$,
\begin{equation}\label{11-25}
|K_t(y_1,x)-K_t(y_2,x)|\le \frac C{t^{n/2}}\frac{|y_1-y_2|^\mu}{t^{\mu/2}}.
\end{equation}
\end{assumption}

\begin{remark}
There exist some operators on $\rn$ whose heat kernels
satisfy Assumption \ref{a-a3}. These operators include Schr\"odinger operators with
non-negative potentials belonging to the reverse H\"older class
(see, for example, \cite{dz02}) and second-order divergence form elliptic operators
with bounded measurable real coefficients (see, for example, \cite{at98}).
\end{remark}

Motivated by \cite[Theorem 1.9]{YYs15}, in this article, we also establish
the following radial maximal function characterization of $H_L^{p(\cdot)}(\rn)$
via showing that \eqref{11-25x} holds true.

\begin{theorem}\label{t-max-rad}
Let $p(\cdot)\in C^{\log}(\rn)$ with $p_+\in(0,1]$ and
$L$ be a linear operator on $L^2(\rn)$ satisfying Assumptions \ref{a-a1}, \ref{a-a2} and \ref{a-a3}.
If $q\in(1,\fz]$, $M\in(\frac n{2}[\frac1{p_-}-1],\fz)\cap\nn$,
then the spaces $H_{L,{\rm at},M}^{p(\cdot),q}(\rn)$, $H_{L,\max}^{p(\cdot)}(\rn)$
and $H_{L,\rad}^{p(\cdot)}(\rn)$ coincide with equivalent quasi-norms.
\end{theorem}

As an immediate consequence of Theorems \ref{t-atom-c}, \ref{t-max-c} and
\ref{t-max-rad}, we have the following conclusion.

\begin{corollary}\label{c-max}
Let $p(\cdot)$, $L$, $q$ and $M$ be as in Theorem \ref{t-max-rad}.
Then, for any $a\in(0,\fz)$ and $\phi$ being as in Definition \ref{d-max-f},
the spaces $H_L^{p(\cdot)}(\rn)$, $H_{L,{\rm at},M}^{p(\cdot),q}(\rn)$,
$H_{L,\max}^{p(\cdot),\phi,a}(\rn)$,
$H_{L,\max}^{p(\cdot),\cf}(\rn)$, $H_{L,\rad}^{p(\cdot)}(\rn)$
and $H_{L,\rad}^{p(\cdot),\cf}(\rn)$ coincide with equivalent quasi-norms.
\end{corollary}

\begin{remark}\label{r-jia1}
Let $\vz:\ \rn\times[0,\fz)\to[0,\fz)$ be a growth function in \cite{ky14}.
D. Yang and S. Yang \cite{YYs15} established
several maximal function characterizations of $H_{\vz,L}(\rn)$, the
\emph{Musielak-Orlicz-Hardy spaces associated with operators} $L$ satisfying the same
assumptions as in the article. Recall that the Musielak-Orlicz space
$L^{\vz}(\rn)$ is defined to be the set of all measurable functions $f$
on $\rn$ such that
$$\|f\|_{L^{\vz}(\rn)}:=\inf\lf\{\lz\in(0,\fz):\
\int_\rn\vz(x,|f(x)|/\lz)\,dx\le1\r\}<\fz,$$
and the space $H_{\vz,L}(\rn)$ is
defined in the same way as $H_L^{p(\cdot)}(\rn)$ with $\|\cdot\|_{\vlp}$
replaced by $\|\cdot\|_{L^\vz(\rn)}$ (see \cite{bckyy}).

Observe that, if
\begin{equation}\label{vz}
\vz(x,t):=t^{p(x)}\quad {\rm for\ all}\quad x\in\rn\quad{\rm and}\quad t\in[0,\fz),
\end{equation}
then $L^{\vz}(\rn)=L^{p(\cdot)}(\rn)$. However, a general
Musielak-Orlicz function $\vz$ satisfying all the assumptions
in \cite{ky14} (and hence \cite{YYs15}) may not have the form
as in \eqref{vz} (see \cite{ky14}). On the other hand, it was proved in
\cite[Remark 2.23(iii)]{yyz14} that there exists a variable exponent function
$p(\cdot)\in C^{\log}(\rn)$,
but $t^{p(\cdot)}$ is not a uniformly Muckenhoupt weight,
which was required in \cite{YYs15}.
Thus, Musielak-Orlicz-Hardy spaces associated with operators in \cite{YYs15}
and variable exponent Hardy spaces associated with operators in this article
do not cover each other.
\end{remark}

This article is organized as follows.

We first show Theorem \ref{t-atom-c} in Section \ref{s2-1}
and then, as an application, we give out the proof of Theorem \ref{t-max-c}
in Section \ref{s2-2}. Finally, in Section \ref{s-2.3}, applying Theorem \ref{t-max-c},
we prove Theorem \ref{t-max-rad}.

We remark that, in the proof of Theorem \ref{t-atom-c},
we borrow some ideas from \cite{jy10,jy11}. Precisely, to establish the atomic
characterization of $H_L^{p(\cdot)}(\rn)$, we need to use the Calder\'on reproducing
formula associated with $L$ (see \eqref{CRF} below) and the atomic decomposition
of the variable tent space $T_2^{p(\cdot)}(\urn)$ established in
\cite[Theorem 2.16]{zyl14} (see also Lemma \ref{l-tent} below). Moreover, we show that
the project operator $\pi_{\Phi,L,M}$ is bounded
from $T_2^{p(\cdot)}(\urn)$ to $H_L^{p(\cdot)}(\rn)$ by
proving that, for any $(p(\cdot),\fz)$-atom
$a$ corresponding to the tent space $T_2^{p(\cdot)}(\urn)$,
$\pi_{\Phi,L,M}(a)$ is a $(p(\cdot),q,M)_L$-atom up to a positive constant multiple
(see Proposition \ref{p-pi} below).
We point out that Lemma \ref{l-key} below obtained by Sawano \cite[Lemma 4.1]{Sa13}
plays a key role in the proofs of Proposition \ref{p-pi} and Theorem \ref{t-atom-c}.

The strategy of the proof of Theorem \ref{t-max-c} is presented in the
following chains of inclusion relations:
\begin{align}\label{11-26}
\quad\lf[H_{L,{\rm at},M}^{p(\cdot),q}(\rn)\cap L^2(\rn)\r]
&\st \lf[H_{L,\max}^{p(\cdot),\cf}(\rn)\cap L^2(\rn)\r]
\st \lf[H_{L,\max}^{p(\cdot),\phi,a}(\rn)\cap L^2(\rn)\r]\\
&\st \lf[H_{L,{\rm at},M}^{p(\cdot),\fz}(\rn)\cap L^2(\rn)\r]
\st \lf[H_{L,{\rm at},M}^{p(\cdot),q}(\rn)\cap L^2(\rn)\r].\noz
\end{align}
The second and the fourth inclusions are obviously.
We prove the first inclusion in \eqref{11-26} via borrowing some ideas from the
proof of \cite[Theorem 1.4]{sy15} and the third one by establishing
a pointwise estimate for the non-tangential maximal function
of any $(p(\cdot),\fz,M)_L$-atom.

The main step in the proof of Theorem \ref{t-max-rad} is to prove that,
for all $f\in L^2(\rn)$,
\begin{equation}\label{jia1}
\lf\|f_{L,\triangledown}^\ast\r\|_{\vlp}
\ls\lf\|f_{L,+}^\ast\r\|_{\vlp}
\end{equation}
via a modified technical based on the proof of \cite[Theorem 2.1.4(b)]{Gra2}.
Indeed, to obtain the inequality \eqref{jia1}, for all $f\in L^2(\rn)$,
we first introduce a maximal function $f_{L,\triangledown}^{\ast,\ez,N}$ of $f$,
where $\ez, N\in(0,\fz)$, which is a truncated version of
the non-tangential maximal function $f_{L,\triangledown}^\ast$
(see \eqref{11-26x} below).
Then, under Assumption \ref{a-a3}, we investigate the relation between
$f_{L,\triangledown}^{\ast,\ez,N}$ and $f_{L,+}^\ast$ in Lemma \ref{l-max4} below,
which is further applied to prove the above inequality.

Here, we point out that the method used in the proof of \eqref{jia1}
is different from that of the case $p(\cdot)\equiv {\rm constant}\in(0,1]$,
which, as a special case, was essentially proved in \cite[Theorem 1.9]{YYs15}.
Indeed, in \cite[Theorem 1.9]{YYs15}, Yang et al. considered the
Musielak-Orlicz Hardy spaces $H_{\vz,L}(\rn)$ associated with the operator $L$
satisfying the same assumptions as in the present article. Moreover, the approach
used in the proof of \cite[Theorem 1.9]{YYs15} strongly depends on the properties
of uniformly Muckenhoupt weights, which are not possessed by $t^{p(\cdot)}$
(see Remark \ref{r-jia1}).

At the end of this section, we make some conventions on notation.
Let $\nn:=\{1,2,\dots\}$, $\zz_+:=\nn\cup\{0\}$ and $\rr_+^{n+1}:=\rn\times(0,\fz)$.
We denote by
$C$ a \emph{positive constant} which is independent of the main
parameters, but may vary from line to line. The \emph{symbol}
$A\ls B$ means $A\le CB$. If $A\ls B$ and $B\ls A$, then we write $A\sim B$.
We use $C_{(\az,\dots)}$ to denote a positive constant depending on the indicated
parameters $\az,\dots$.
If $E$ is a subset of $\rn$, we denote by $\chi_E$ its
\emph{characteristic function} and by $E^\complement$ the set $\rn\backslash E$.
For $a \in {\mathbb R}$, $\lfloor a \rfloor$ denotes the largest integer $m$
such that $m \le a$. For all $x\in\rn$ and $r\in(0,\fz)$, denote by $Q(x,r)$
the cube centered at $x$ with side length $r$, whose sides are parallel to the axes
of coordinates, and by $B(x,r)$ the ball, namely,
$B(x,r):=\{y\in\rn:\ |y-x|<r\}$.
For each cube $Q\st\rn$ and $a\in(0,\fz)$,
we use $x_Q$ to denote the center of $Q$ and $\ell(Q)$ the side length
of $Q$, and we also denote by $aQ$
the cube concentric with $Q$ having the side length $a\ell(Q)$.

\section{Proof of Theorem \ref{t-atom-c}\label{s2-1}}
\hskip\parindent
To prove Theorem \ref{t-atom-c}, we first recall some notions about the
variable exponent tent space introduced in \cite{zyl14}.

Let $p(\cdot)\in\cp(\rn)$.
For all measurable functions $g$
on $\rr_+^{n+1}$ and $x\in\rn$, define
$$\ct(g)(x)
:=\lf\{\int_{\Gamma(x)}|g(y,t)|^2\,
\frac{dy\,dt}{t^{n+1}}\r\}^{1/2},$$
where $\Gamma(x)$ is as in \eqref{1.2x}.
Then the \emph{variable exponent tent space}
$T_2^{p(\cdot)}(\urn)$ is defined to be the set of all
measurable functions $g$ on $\rr_+^{n+1}$ such that
$\|g\|_{T^{p(\cdot)}_2(\rr_+^{n+1})}
:=\|\ct(g)\|_{\vlp}<\fz$.
Recall that, for any $q\in(0,\fz)$ being a constant exponent,
the \emph{tent space} $T_2^q(\rr_+^{n+1})$ was introduced in \cite{cms85}, which
is defined to be the set of all measurable functions $g$ on $\urn$ such that
$\|g\|_{T_2^q(\rr_+^{n+1})}:=\|\ct(g)\|_{L^q(\rn)}<\fz$.
Moreover, if $g\in T_2^2(\urn)$, then we easily know that
\begin{equation*}
\|g\|_{T_2^2(\urn)}=\lf\{\int_{\urn}|g(x,t)|^2\,\frac{dxdt}t\r\}^{\frac12}.
\end{equation*}

Let $q\in(1,\fz)$ and $p(\cdot)\in\cp(\rn)$.
Recall that a measurable function $a$ on $\rr_+^{n+1}$ is called a
$(p(\cdot),q)$-\emph{atom} if $a$ satisfies that
$\supp a\subset\wh Q$ for some cube $Q\subset\rn$ and
$$\|a\|_{T_2^q(\rr_+^{n+1})}\le |Q|^{1/q}\|\chi_Q\|_{\vlp}^{-1},$$
here and hereafter, for any cube $Q\st\rn$, $\wh Q$ denotes the
\emph{tent} over $Q$, namely,
$$\wh Q:=\lf\{(y,t)\in\urn:\ B(y,t)\st Q\r\}.$$
Furthermore, if $a$ is a $(p(\cdot),q)$-atom for all
$q\in(1,\fz)$, then $a$ is called a $(p(\cdot),\fz)$-\emph{atom}. We point
out that the notion of  $(p(\cdot),\fz)$-atoms was introduced in \cite{zyl14}.

For any $p(\cdot)\in\cp(\rn)$, $\{\lz_j\}_{j\in\nn}\st\cc$
and $\{Q_j\}_{j\in\nn}$ of cubes in $\rn$, let
\begin{equation*}
\ca\lf(\{\lz_j\}_{j\in\nn},\{Q_j\}_{j\in\nn}\r)
:=\lf\|\lf\{\sum_{j\in\nn}\lf[\frac{|\lz_j|\chi_{Q_j}}{\|Q_j\|_{\vlp}}
\r]^{\underline{p}}\r\}^{\frac 1{\underline{p}}}\r\|_{\vlp},
\end{equation*}
where $\underline{p}:=\min\{1,p_-\}$.

The following atomic characterization of the space $T_2^{p(\cdot)}(\urn)$
was obtained in \cite[Corollary 3.7]{yz15}.

\begin{lemma}\label{l-tent}
Let $p(\cdot)\in C^{\log}(\rn)$.
Then $f\in T^{p(\cdot)}_2(\urn)$ if and only if there exist
sequences $\{\lz_j\}_{j\in\nn}\subset\cc$ and $\{a_j\}_{j\in\nn}$
of $(p(\cdot),\fz)$-atoms such that, for almost every
$(x,t)\in\urn$,
\begin{equation}\label{tent-de}
f(x,t)=\sum_{j\in\nn}\lz_ja_j(x,t)
\end{equation}
and
\begin{equation*}
\int_\rn\lf\{\sum_{j\in\nn}
\lf[\frac{\lz_j\chi_{Q_j}}{\|\chi_{Q_j}\|_{\vlp}}\r]^{\underline{p}}
\r\}^{\frac{p(x)}{\underline{p}}}\,dx<\fz,
\end{equation*}
where, for each $j$, $Q_j$ denotes the cube appearing in the support of $a_j$;
moreover, for all $f\in T_2^{p(\cdot)}(\urn)$,
$
\|f\|_{T_2^{p(\cdot)}(\urn)}\sim\ca(\{\lz_j\}_{j\in\nn},\{Q_j\}_{j\in\nn})
$
with the implicit equivalent positive constants independent of $f$.
\end{lemma}

In what follows, let $T_{2,c}^{p(\cdot)}(\urn)$ and $T_{2,c}^q(\urn)$ with
$q\in(0,\fz)$ be the sets of all functions in $T_2^{p(\cdot)}(\urn)$
, respectively, $T_2^q(\urn)$ with compact supports.

\begin{remark}\label{r-ten1}
Let $p(\cdot)\in C^{\log}(\rn)$.

(i) It is known that $T_{2,c}^{p(\cdot)}(\urn)\subset T_{2,c}^2(\urn)$ as sets
(see \cite[Proposition 3.9]{yz15}).

(ii) By \cite[Corollary 3.4]{yz15}, we know that, for all $f\in T_2^{p(\cdot)}(\urn)$,
the decomposition \eqref{tent-de} also holds true in $T_2^{p(\cdot)}(\urn)$,
which, in the case that $p(\cdot)\equiv {\rm constant}\in(0,\fz)$, was proved
by Jiang and Yang in \cite[Proposition 3.1]{jy10}.
\end{remark}

For a non-negative self-adjoint operator $L$ on $L^2(\rn)$,
denoted by $E_L$ the spectral measure associated with $L$.
Then, for any bounded Borel measurable function $F:\ [0,\fz)\to\cc$,
the operator $F(L):\ L^2(\rn)\to L^2(\rn)$ is defined by the formula
\begin{equation*}
F(L):=\int_0^\fz F(\lz)dE_L(\lz).
\end{equation*}

Let $\phi_0\in \cs(\rr)$ be a given even function and $\supp \phi_0\st(-1,1)$.
Assume that $\Phi$ denotes the \emph{Fourier transform} of $\phi_0$, namely,
for all $\xi\in\rn$, $\Phi(\xi):=\int_\rn \phi_0(x)e^{-ix\cdot\xi}\,dx$.
For all $f\in L^2(\urn)$ having compact support and $x\in \rn$, define
\begin{equation*}
\pi_{\Phi,L,M}(f)(x):=C_{(\Phi,M)}\int_0^\fz(t^2L)^{M+1}\Phi(t\sqrt L)(f(\cdot,t))(x)\,
\frac{dt}t,
\end{equation*}
where $C_{(\Phi,M)}$ is the positive constant such that
\begin{equation}\label{2.2x}
1=C_{(\Phi,M)}\int_0^\fz t^{2(M+1)}\Phi(t)t^2e^{-t^2}\,\frac{dt}t.
\end{equation}

We then have the following lemma, which is a part of \cite[Lemma 3.5]{hlmmy}.

\begin{lemma}\label{l-lemma1}
Let $\phi_0\in \cs(\rr)$ be an even function and $\supp \phi_0\st(-1,1)$. Assume that
$\Phi$ denotes the Fourier transform of $\phi_0$. Then, for any $k\in\zz_+$, the kernels
$\{K_{(t^2L)^k\Phi(t\sqrt L)}\}_{t>0}$ of the operators
$\{(t^2L)^k\Phi(t\sqrt L)\}_{t>0}$ satisfy that there exists a positive constant
$C$ such that, for all $t\in(0,\fz)$ and $x,\,y\in\rn$,
\begin{equation*}
\supp \lf(K_{(t^2L)^k\Phi(t\sqrt L)}\r)\st \{(x,y)\in\rn\times\rn:\ |x-y|\le t\}.
\end{equation*}
\end{lemma}

\begin{remark}\label{r-pi}
The operator $\pi_{\Phi,L,M}$, initially defined on $T_{2,c}^{2}(\urn)$,
extends to a bounded linear operator from $T_2^2(\urn)$ to $L^2(\rn)$
(see \cite[Proposition 4.2(ii)]{jy11}).
\end{remark}

Moreover, we have the following conclusion.

\begin{proposition}\label{p-pi}
Let $L$ and $p(\cdot)$ be as in Definition \ref{d-hardys}.

{\rm(i)} Let $M\in \nn$ and $a$ be a $(p(\cdot),\fz)$-atom. Then $\pi_{\Phi,L,M}(a)$
is a $(p(\cdot),\fz,M)_L$-atom up to a positive constant multiple.

{\rm(ii)} The operator $\pi_{\Phi,L,M}$, initially defined on $T_{2,c}^{p(\cdot)}(\urn)$,
extends to a bounded linear operator from $T_2^{p(\cdot)}(\urn)$
to $H_L^{p(\cdot)}(\rn)$.
\end{proposition}

The proof of Proposition \ref{p-pi} depends on
the following several lemmas.
The following Fefferman-Stein vector-valued inequality
of the Hardy-Littlewood maximal operator $\cm$
on the space $\vlp$ was obtained in \cite[Corollary 2.1]{cfmp06}.

\begin{lemma}\label{l-hlmo}
Let $r\in(1,\fz)$ and $p(\cdot)\in C^{\log}(\rn)$.
If $p_-\in(1,\fz)$ with $p_-$ as in \eqref{2.1x}, then there exists a positive
constant $C$ such that, for all sequences $\{f_j\}_{j=1}^\fz$ of measurable
functions,
$$\lf\|\lf\{\sum_{j=1}^\fz \lf[\cm (f_j)\r]^r\r\}^{1/r}\r\|_{\vlp}
\le C\lf\|\lf(\sum_{j=1}^\fz|f_j|^r\r)^{1/r}\r\|_{\vlp}.$$
\end{lemma}

\begin{remark}\label{r-vlp-m}
Let $p(\cdot)\in C^{\log}(\rn)$ and $p_-\in(1,\fz)$. Then there exists a positive
constant $C$ such that, for all $f\in \vlp$,
$\|\cm(f)\|_{\vlp}\le C\|f\|_{\vlp}$
(see, for example, \cite[Theorem 4.3.8]{dhr11}).
\end{remark}

\begin{remark}\label{r-hlmo}
Let $k\in\nn$ and $p(\cdot)\in C^{\log}(\rn)$.
Then, by Lemma \ref{l-hlmo} and the fact that, for all cubes $Q\subset\rn$,
$r\in(0,p_-)$,
$\chi_{2^kQ}\le 2^{kn/r} [\cm(\chi_{Q})]^{1/r}$, we conclude that there exists
a positive constant $C$ such that,
for any $\{\lz_j\}_{j\in\nn}\subset\cc$
and cubes $\{Q_j\}_{j\in\nn}$ of $\rn$,
$$\lf\|\lf\{\sum_{j\in\nn}\lf[\frac{|\lz_j|\chi_{2^kQ_j}}{\|\chi_{Q_j}\|_{\vlp}}
\r]^{p_-}\r\}^{1/p_-}\r\|_\vlp\le
C2^{kn/r}\ca(\{\lz_j\}_{j\in\nn},\{Q_j\}_{j\in\nn}).$$
\end{remark}

The following lemma is just \cite[Lemma 2.6]{zyl14}.
\begin{lemma}\label{l-bigsball}
Let $p(\cdot)\in C^{\log}(\rn)$.
Then there exists a positive
constant $C$ such that, for all cubes $Q_1$ and $Q_2$ of $\rn$ with $Q_1\st Q_2$,
\begin{equation*}
C^{-1}\lf(\frac{|Q_1|}{|Q_2|}\r)^{1/p_-}
\le\frac{\|\chi_{Q_1}\|_{\vlp}}{\|\chi_{Q_2}\|_{\vlp}}
\le C\lf(\frac{|Q_1|}{|Q_2|}\r)^{1/p_+},
\end{equation*}
where $p_-$ and $p_+$ are as in \eqref{2.1x}.
\end{lemma}

We also need the following useful lemma, which is just
\cite[Lemma 4.1]{Sa13} and plays a key role in the present article.

\begin{lemma}\label{l-key}
Let $p(\cdot)\in C^{\log}(\rn)$ and $q\in[1,\fz)\cap(p_+,\fz)$,
where $p_+$ is as in \eqref{2.1x}.
Then there exists a positive constant
$C$ such that, for all sequences $\{Q_j\}_{j\in\nn}$ of cubes,
$\{\lz_j\}_{j\in\nn}\st\cc$ and functions
$\{a_j\}_{j\in\nn}$ satisfying that,  for each $j\in\nn$,
$\supp a_j\st Q_j$ and $\|a_j\|_{L^q(\rn)}\le |Q_j|^{1/q}$,
\begin{equation*}
\lf\|\lf(\sum_{j=1}^\fz|\lz_ja_j|^{\underline{p}}\r)^\frac1{\underline{p}}\r\|_{\vlp}
\le C\lf\|\lf(\sum_{j=1}^\fz|\lz_j\chi_{Q_j}|^{\underline{p}}
\r)^\frac1{\underline{p}}\r\|_{\vlp}.
\end{equation*}
\end{lemma}

We are now ready to prove Proposition \ref{p-pi}.

\begin{proof}[Proof of Proposition \ref{p-pi}]
We first prove (i). Let $a$ be a $(p(\cdot),\fz)$-atom associated with some cube
$Q\st\rn$. Let
$$b:=C_{(\Phi,M)}\int_0^\fz t^{2(M+1)}L\Phi(t\sqrt L)(a(\cdot,t))\,\frac{dt}t,$$
where $C_{(\Phi,M)}$ is as in \eqref{2.2x}.
Then $\pi_{\Phi,L,M}(a)=L^M(b)$. By Lemma \ref{l-lemma1} and the fact that
$\supp a\st Q$, we easily know that $\supp L^kb\st \sqrt nQ$ for each
$k\in\{0,\,1\,\dots,\,M\}$. On the other hand, by \cite[Lemma 2]{cms85}
and the H\"older inequality, we find that,
for any $q\in(1,\fz)$ and $h\in L^2(Q)\cap L^{q'}(Q)$ with $q':=\frac q{q-1}$,
\begin{align*}
&\lf|\int_\rn ([\ell(Q)]^2L)^kb(x)h(x)\,dx\r|\\
&\hs\ls [\ell(Q)]^{2M}\int_\rn\int_0^{\ell(Q)}\lf|a(y,t)(t^2L)^{k+1}
\Phi(t\sqrt L)(h)(y)\r|\,\frac{dt}tdy\\
&\hs\ls [\ell(Q)]^{2M}\int_\rn\lf\{\int_{\Gamma(x)}|a(y,t)|
\lf|(t^2L)^{k+1}\Phi(t\sqrt L)h(y)\r|\,\frac{dydt}{t^{n+1}}\r\}dx\\
&\hs\ls [\ell(Q)]^{2M}\|a\|_{T_2^q(\urn)}\lf\|\wz S_L^k(h)\r\|_{L^{q'}(\rn)}
\ls\frac{[\ell(Q)]^{2M+n/q}}{\|\chi_Q\|_{\vlp}}\|h\|_{L^{q'(\rn)}},
\end{align*}
where
$$\wz S_L^k(h)(x):=\lf\{\int_{\Gamma(x)}\lf|(t^2L)^{k+1}\Phi(t\sqrt L)h(y)\r|^2\,
\frac{dydt}{t^{n+1}}\r\}^{1/2}$$
with $\Gamma(x)$ as in \eqref{1.2x},
which is bounded on $L^r(\rn)$ with $r\in(1,\fz)$
(see, for example, \cite[Lemma 5.3]{bckyy}).
Therefore, $\pi_{\Phi,L,M}(a)$ is a
$(p(\cdot),\fz,M)_L$-atom up to a positive constant multiple and hence
the proof of (i) is completed.

Next, we show (ii).
Let $f\in T_{2,c}^{p(\cdot)}(\urn)$. Then, by Remark \ref{r-ten1}(i), we know
that $f\in T_{2,c}^2(\urn)$ and hence, due to Remark \ref{r-pi},
$\pi_{\Phi,L,M}$ is well defined on $T_{2,c}^{p(\cdot)}(\urn)$.
From this, combined with Lemma \ref{l-tent} and Remark \ref{r-ten1}(ii),
we deduce that $f=\sum_{j\in\nn}\lz_ja_j$ in both $T_2^{p(\cdot)}(\urn)$ and $T_2^2(\urn)$,
where $\{\lz_j\}_{j\in\nn}\st\cc$ and $\{a_j\}_{j\in\nn}$ are $(p(\cdot),\fz)$ atoms
associated cubes $\{Q_j\}_{j\in\nn}$ of $\rn$ satisfying
\begin{equation}\label{area-esti-y}
\ca(\{\lz_j\}_{j\in\nn},\{Q_j\}_{j\in\nn})\ls \|f\|_{T_2^{p(\cdot)}(\urn)};
\end{equation}
moreover
\begin{equation*}
\pi_{\Phi,L,M}(f)=\sum_{j\in\nn}\lz_j\pi_{\Phi,L,M}(a_j)=:\sum_{j\in\nn}\lz_j\az_j
\quad {\rm in}\quad L^2(\rn).
\end{equation*}
Obviously, for any $j\in\nn$, $\az_j$ is a $(p(\cdot),\fz,M)_L$-atom
up to a positive constant multiple by (i). Since $S_L$ is bounded on $L^2(\rn)$
(see \eqref{SL-bounded}),
it follows that, for almost every $x\in\rn$,
$$S_L(\pi_{\Phi,L,M}(f))(x)\le \sum_{j\in\nn}|\lz_j|S_L(\az_j)(x).$$
Thus, we have
\begin{align*}
&\lf\|S_L(\pi_{\Phi,L,M}(f))\r\|_\vlp\\
&\hs\le \lf\|\sum_{j\in\nn}|\lz_j|S_L(\alpha_j)\chi_{4\sqrt nQ_j}\r\|_\vlp
+\lf\|\sum_{j\in\nn}|\lz_j|S_L(\alpha_j)\chi_{(4\sqrt nQ_j)^\complement}\r\|_\vlp
=:{\rm I}+{\rm II}.
\end{align*}

Observe that, by \eqref{SL-bounded}, we find that, for any $q\in(1,\fz)$ and $j\in\nn$,
\begin{equation*}
\|S_L(\alpha_j)\|_{L^q(\rn)}\ls \|\alpha_j\|_{L^q(\rn)}
\ls \frac{|4\sqrt nQ_j|^{1/q}}{\|\chi_{Q_j}\|_{\vlp}}.
\end{equation*}
By this, Lemma \ref{l-key}, Remark \ref{r-hlmo} and \eqref{area-esti-y},
we conclude that
\begin{align*}
{\rm I}\ls \lf\|\lf\{\sum_{j\in\nn}\lf[\frac{|\lz_j|\chi_{4\sqrt nQ_j}}{\|\chi_{Q_j}\|_\vlp}
\r]^{p_-}\r\}^{1/p_-}\r\|_{\vlp}
\ls\ca(\{\lz_j\}_{j\in\nn},\{Q_j\}_{j\in\nn})\ls \|f\|_{T_2^{p(\cdot)}(\urn)}.
\end{align*}

To estimate II, we first claim that, for all $\delta\in(n[1/p_--1],2M)$,
$j\in\nn$ and $x\in (4\sqrt nQ_j)^\complement$,
\begin{equation}\label{area-esti0}
S_L(\alpha_j)(x)\ls \frac{[\ell(Q_j)]^{n+\delta}}{|x-x_{Q_j}|^{n+\delta}}
\frac1{\|\chi_{Q_j}\|_\vlp}.
\end{equation}
Indeed, for all $j\in\nn$ and $x\in (4\sqrt nQ_j)^\complement$,
\begin{align}\label{area-esti}
\lf[S_L(a_j)(x)\r]^2
&=\int_0^{\ell(Q_j)}\int_{|y-x|<t}\lf|t^2Le^{-t^2L}(\az_j)(y)
\r|^2\,\frac{dydt}{t^{n+1}}
+\int_{\ell(Q_j)}^\fz\int_{|y-x|<t}\cdots\\
&=:{\rm II}_1(x)+{\rm II}_2(x).\noz
\end{align}
Notice that
$t^2Le^{-t^2L}=(-r\frac{de^{-rL}}{dr})_{r=t^2}$. It follows from
\eqref{partial1} that
\begin{align}\label{area-esti-x}
{\rm II}_1(x)&=\int_0^{\ell(Q_j)}\int_{|y-x|<t}
\lf|\lf(r\frac{de^{-rL}}{dr}\r)_{r=t^2}(\az_j)(y)\r|^2\,\frac{dydt}{t^{n+1}}\\
&\ls \int_0^{\ell(Q_j)}\int_{|y-x|<t}\lf[\int_{Q_j}\frac1{t^n}e^{-\frac{|z-y|^2}{t^2}}
|\az_j(z)|\,dz\r]^2\,\frac{dydt}{t^{n+1}}\noz\\
&\ls \int_0^{\ell(Q_j)}\int_{|y-x|<t}\lf[\int_{Q_j}
\frac{t^\delta}{(t+|z-y|)^{n+\delta}}
|\az_j(z)|\,dz\r]^2\,\frac{dydt}{t^{n+1}}.\noz
\end{align}
Since, for all $x\in (4\sqrt nQ_j)^\complement$, $t\in (0,\fz)$,
$|y-x|<t$ and $z\in Q_j$, we have
\begin{equation}\label{relation1}
t+|z-y|\ge |x-z|\ge |x-x_{Q_j}|-|x_{Q_j}-z|\ge |x-x_{Q_j}|/2.
\end{equation}
By this, \eqref{area-esti-x} and the H\"older inequality, we further find that
\begin{align}\label{area-esti1}
{\rm II}_1(x)&\ls \frac{|Q_j|^{2(1-1/q)}}{|x-x_{Q_j}|^{2(n+\delta)}}
\|\az_j\|_{L^q(\rn)}^2\int_0^{\ell(Q_j)}t^{2\delta}\frac{dt}t\\
&\ls \frac{[\ell(Q_j)]^{2(n+\delta)}}{|x-x_{Q_j}|^{2(n+\delta)}}
\frac1{\|\chi_{Q_j}\|_\vlp^2}.\noz
\end{align}
On the other hand, by the proof of (i), we know that, for each $j\in\nn$,
there exists $b_j\in \cd(L^M)$ such that $\az_j=L^M(b_j)$ and
$$\|b_j\|_{L^q(\rn)}\ls [\ell(Q_j)]^{2M}|Q_j|^{1/q}
\|\chi_{Q_j}\|_{\vlp}^{-1}.$$
From this, \eqref{partial1}, \eqref{relation1} and the fact that
$$t^2L^{M+1}e^{-t^2L}=(-1)^{M+1}t^{-2M}\lf(r^{M+1}\frac{d^{M+1}e^{-rL}}{dr^{M+1}}
\r)_{r=t^2},$$
we deduce that
\begin{align}\label{area-esti2}
{\rm II}_2(x)
&\sim \int_{\ell(Q_j)}^\fz\int_{|y-x|<t}\lf|\lf(r^{M+1}\frac{d^{M+1}e^{-rL}}
{dr^{M+1}}\r)_{r=t^2}(b_j)(y)\r|^2\,\frac{dydt}{t^{n+4M+1}}\\
&\ls \int_{\ell(Q_j)}^\fz\int_{|y-x|<t}\lf[\int_{Q_j}\frac1{t^n}
e^{-\frac{|z-y|^2}{t^2}}|b_j(z)|\,dz\r]^2\frac{dydt}{t^{n+4M+1}}\noz\\
&\ls\frac{|Q_j|^{2(1-1/q)}}{|x-x_{Q_j}|^{2(n+\delta)}}
\|b_j\|_{L^q(\rn)}^2\int_{\ell(Q_j)}^\fz t^{2\delta}\frac{dt}{t^{4M+1}}\noz\\
&\ls \frac{[\ell(Q_j)]^{2(n+\delta)}}{|x-x_{Q_j}|^{2(n+\delta)}}
\frac1{\|\chi_{Q_j}\|_\vlp^2}.\noz
\end{align}

Combining \eqref{area-esti}, \eqref{area-esti1} and \eqref{area-esti2}, we conclude
that \eqref{area-esti0} holds true, which completes the proof of the above claim.

Now, let $r\in(0,p_-)$ be such that $\delta\in(n[1/r-1],2M)$.
Then, from the above claim, Remarks \ref{r-vlp} and \ref{r-hlmo}, and \eqref{area-esti-y},
we deduce that
\begin{align*}
{\rm II}&\ls\lf\|\sum_{j\in\nn}\sum_{k\in\nn}|\lz_j|S_L(\az_j)
\chi_{2^{k+2}\sqrt nQ_j\backslash(2^{k+1}\sqrt nQ_j)}\r\|_\vlp\\
&\ls \lf\{\sum_{k\in\nn}2^{-k(n+\delta)p_-}
\lf\|\sum_{j\in\nn}\lf[\frac{|\lz_j|\chi_{2^{k+2}\sqrt nQ_j}}{\|\chi_{Q_j}\|_{\vlp}}
\r]^{p_-}\r\|_{L^{{p(\cdot)}/{p_-}}(\rn)}\r\}^{1/p_-}\\
&\ls \lf\{\sum_{k\in\nn}2^{-k(n+\delta)p_-}2^{knp_-/r}
\lf\|\lf\{\sum_{j\in\nn}\lf[\frac{|\lz_j|\chi_{Q_j}}{\|\chi_{Q_j}\|_{\vlp}}
\r]^{p_-}\r\}^{1/p_-}\r\|_{\vlp}^{p_-}
\r\}^{1/p_-}\\
&\ls \|f\|_{T_2^{p(\cdot)}(\urn)}
\lf\{\sum_{k\in\nn}2^{-k(n+\delta-n/r)}\r\}^{1/p_-}
\sim \|f\|_{T_2^{p(\cdot)}(\urn)}.
\end{align*}
This, together with the estimate of I, implies that
\begin{equation}\label{SL-pi}
\|S_L(\pi_{\Phi,L,M}(f))\|_{\vlp}\ls \|f\|_{T_2^{p(\cdot)}(\urn)}
\end{equation} and hence (ii) holds true,
which completes the proof of Proposition \ref{p-pi}.
\end{proof}

We now prove Theorem \ref{t-atom-c} by using Proposition \ref{p-pi}.

\begin{proof}[Proof of Theorem \ref{t-atom-c}]
We first show
\begin{equation}\label{atom-c1}
H_{L,{\rm at},M}^{p(\cdot),q}(\rn)\st H_L^{p(\cdot)}(\rn).
\end{equation}
Let $f\in H_{L,{\rm at},M}^{p(\cdot),q}(\rn)\cap L^2(\rn)$. Then,
by Definition \ref{d-atom}, we know that $f$ has a
representation $f=\sum_{j\in\nn}\lz_j\alpha_j$ in $L^2(\rn)$,
where $\{\lz_j\}_{j\in\nn}\st\cc$ and $\{\alpha_j\}_{j\in\nn}$
are $(p(\cdot),q,M)_L$-atoms such that
$\cb(\{\lz_j\alpha_j\}_{j\in\nn})\ls \|f\|_{H_{L,{\rm at},M}^{p(\cdot),q}(\rn)}$.
By an argument similar to that used in the proof of Proposition \ref{p-pi}(ii),
we conclude that
$\|S_L(f)\|_{\vlp}\ls \|f\|_{H_{L,{\rm at},M}^{p(\cdot),q}(\rn)}$, which implies that
$$\lf[H_{L,{\rm at},M}^{p(\cdot),q}(\rn)\cap L^2(\rn)\r]\st H_L^{p(\cdot)}(\rn)$$
and hence \eqref{atom-c1} holds true by Remark \ref{r-dense}.

Conversely, we need to show
\begin{equation}\label{CRF}
H_L^{p(\cdot)}(\rn)\st H_{L,{\rm at},M}^{p(\cdot),q}(\rn).
\end{equation}
Let $f\in H_L^{p(\cdot)}(\rn)\cap L^2(\rn)$. Then, by the functional
calculi for $L$, we know that
$$f=C_{(\Phi,M)}\int_0^\fz(t^2L)^{M+1}\Phi(t\sqrt L)(t^2Le^{-t^2L}f)\,\frac{dt}t
=\pi_{(\Phi,L,M)}(t^2Le^{-t^2L}f)
\quad {\rm in}\quad L^2(\rn),$$
where $C_{(\Phi,M)}$ is as in \eqref{2.2x}.
Since $S_L(f)\in \vlp\cap L^2(\rn)$, it follows that
$t^2Le^{-t^2L}(f)\in T_2^{p(\cdot)}(\urn)\cap T_2^2(\urn)$.
Thus, from Lemma \ref{l-tent} and Proposition \ref{p-pi}(ii),
we deduce that there exist sequences
$\{\lz_j\}_{j\in\nn}\st\cc$ and $\{a_j\}_{j\in\nn}$ of
$(p(\cdot),\fz)$-atoms such that
\begin{align*}
f=\pi_{\Phi,L,M}(t^2Le^{-t^2L}f)=\sum_{j\in\nn}\lz_j\pi_{\Phi,L,M}(a_j)
=:\sum_{j\in\nn}\lz_j\az_j\  {\rm in\ both}\ L^2(\rn)\ {\rm and}\ H_L^{p(\cdot)}(\rn),
\end{align*}
and
$$\cb(\{\lz_ja_j\}_{j\in\nn})\ls \lf\|t^2Le^{-t^2L}f\r\|_{T_2^{p(\cdot)}(\urn)}
\sim\|f\|_{H_L^{p(\cdot)}(\rn)}.$$
On the other hand, by Proposition \ref{p-pi}(i), we know that, for $j\in\nn$, $\az_j$
is a $(p(\cdot),\fz,M)_L$-atom up to a positive constant multiple.
Therefore,
$f\in H_{L,{\rm at},M}^{p(\cdot),\fz}(\rn)\st H_{L,{\rm at},M}^{p(\cdot),q}(\rn)$,
which implies that \eqref{CRF} holds true and hence completes
the proof of Theorem \ref{t-atom-c}.
\end{proof}

\section{Proof of Theorem \ref{t-max-c}\label{s2-2}}
\hskip\parindent
In this section, we prove Theorem \ref{t-max-c}.
We first recall the following notion.

For a given Borel measurable function $F$ on $\urn$, the \emph{non-tangential
maximal function} of $F$ with aperture $\az\in(0,\fz)$ is defined by setting, for all $x\in\rn$,
\begin{equation}\label{11-26z}
M_\az(F)(x):=\sup_{t\in(0,\fz),\, |y-x|<\az t}|F(y,t)|.
\end{equation}

\begin{lemma}\label{l-lemma0}
Let $p(\cdot)\in C^{\log}(\rn)$ and $\az_1,\ \az_2\in(0,\fz)$. If $\lz\in(n/p_-,\fz)$,
then there exists a positive constant $C$ such that,
for any Borel measurable function $F$ on $\urn$,
\begin{equation}\label{max-f1}
\lf\|M_{\az_1}(F)\r\|_{\vlp}\le
C\lf(1+\frac{\az_1}{\az_2}\r)^{\lz}
\lf\|M_{\az_2}(F)\r\|_{\vlp}.
\end{equation}
\end{lemma}

\begin{proof}
For any $\az\in(0,\fz)$ and $\lz\in(n/p_-,\fz)$, let
\begin{equation}\label{11-26y}
N_\lz^\az(F)(x):=\sup_{t\in(0,\fz),\,y\in\rn}
\lf|F(y,t)\r|\lf(1+\frac{|x-y|}{\az t}\r)^{-\lz}.\
\end{equation}
Then it is easy to see that
$M_\az(F)(x)\ls N_\lz^\az(F)(x)$ for all $x\in\rn$.

Therefore, to prove \eqref{max-f1}, it suffices to show that,
for any $\az_1,\ \az_2\in(0,\fz)$,
\begin{equation}\label{max-f1-x}
\lf\|N_\lz^{\az_1}(F)\r\|_{\vlp}
\ls \lf(1+\frac{\az_1}{\az_2}\r)^\lz
\lf\|M_{\az_2}(F)\r\|_{\vlp}.
\end{equation}
To prove this, we first notice that, for any $t\in(0,\fz)$, $x,\ y\in\rn$ and all
$z\in B(x-y,\az_2 t)$,
$$|F(x-y,t)|\le M_{\az_2}(F)(z).$$
Then, since $B(x-y,\az_2 t)\st B(x,|y|+\az_2t)$, it follows that
\begin{align*}
\lf|F(x-y,t)\r|^{\frac n\lz}
&\le \frac1{|B(x-y,\az_2t)|}
\int_{B(x,|y|+\az_2t)}\lf|M_{\az_2}(F)(z)\r|^{\frac n\lz}\,dz\\
&\le \frac{|B(x,|y|+\az_2t)|}{|B(x-y,\az_2t)|}
\cm\lf(\lf[M_{\az_2}(F)\r]^{n/\lz}\r)(x)\\
&\ls\lf(1+\frac{\az_1}{\az_2}\r)^n\lf(1+\frac{|y|}{\az_1t}\r)^n
\cm\lf(\lf[M_{\az_2}(F)\r]^{n/\lz}\r)(x).
\end{align*}
Thus, we conclude that, for all $t\in(0,\fz)$ and $x,\ y\in\rn$,
\begin{equation*}
\lf|F(x-y,t)\r|\lf(1+\frac{|y|}{\az_1t}\r)^{-\lz}
\ls \lf(1+\frac{\az_1}{\az_2}\r)^\lz
\lf\{ \cm\lf(\lf[M_{\az_2}(F)\r]^{n/\lz}\r)(x)
\r\}^{\frac \lz n},
\end{equation*}
which further implies that
\begin{equation*}
N_\lz^{\az_1}(F)(x)\ls \lf(1+\frac{\az_1}{\az_2}\r)^\lz
\lf\{ \cm\lf(\lf[M_{\az_2}(F)\r]^{n/\lz}\r)(x)
\r\}^{\frac \lz n}.
\end{equation*}
From this, Remark \ref{r-vlp-m} and the fact that
$\lz\in(n/p_-,\fz)$, we deduce that \eqref{max-f1-x} holds true, which completes
the proof of Lemma \ref{l-lemma0}.
\end{proof}

\begin{remark}
When $p(\cdot)\equiv {\rm constant}\in(0,\fz)$, Lemma \ref{l-lemma0} was
established by Calder\'on and Torchinsky in \cite[Theorem 2.3]{ct75}.
\end{remark}

By Lemma \ref{l-lemma0}, we immediate obtain the following conclusion, the details
being omitted.

\begin{corollary}\label{c-lemma3}
Let $L$ be as in Theorem \ref{t-max-c}, $p(\cdot)\in C^{\log}(\rn)$,
$\az_1,\ \az_2\in(0,\fz)$ and
$\vz\in\cs(\rr)$ be an even function with $\vz(0)=1$.
If $\lz\in(n/p_-,\fz)$, then there exists a positive constant $C$ such that,
for all $f\in L^2(\rn)$,
\begin{equation}\label{11-29}
\lf\|\vz_{L,\triangledown,\az_1}^\ast(f)\r\|_{\vlp}\le
C\lf(1+\frac{\az_1}{\az_2}\r)^{\lz}
\lf\|\vz_{L,\triangledown,\az_2}^\ast(f)\r\|_{\vlp}.
\end{equation}
\end{corollary}

We also have the following technical lemma.

\begin{lemma}\label{l-lemma2}
Let $L$ be as in Theorem \ref{t-max-c}, $p(\cdot)\in C^{\log}(\rn)$,
$\psi_1,\ \psi_2\in\cs(\rr)$ be even functions with $\psi_1(0)=1=\psi_2(0)$ and
$\az_1,\ \az_2\in(0,\fz)$.
Then there exists a positive constant $C\in(0,\fz)$, depending on $\psi_1$, $\psi_2$,
$\az_1$ and $\az_2$, such that, for all $f\in L^2(\rn)$,
\begin{equation}\label{max-f0}
\lf\|(\psi_1)_{L,\triangledown,\az_1}^\ast(f)\r\|_{\vlp}\le C\lf\|(\psi_2)_{L,\triangledown,\az_2}^\ast(f)\r\|_{\vlp}.
\end{equation}
\end{lemma}

\begin{proof}
Let $\psi:=\psi_1-\psi_2$. Then, by Remark \ref{r-vlp}, we have
$$\lf\|(\psi_1)_{L,\triangledown,\az_1}^\ast(f)\r\|_{\vlp}
\ls \lf\|\psi_{L,\triangledown,\az_1}^\ast(f)\r\|_{\vlp}+\lf\|(\psi_2)_{L,\triangledown,\az_1}^\ast(f)\r\|_{\vlp}.$$
Thus, to prove \eqref{max-f0}, by Corollary \ref{c-lemma3}, it suffices to show that
\begin{equation}\label{max-f2}
\lf\|\psi_{L,\triangledown,\az_1}^\ast(f)\r\|_{\vlp}
\ls \lf\|(\psi_2)_{L,\triangledown,\az_2}^\ast(f)\r\|_{\vlp}.
\end{equation}
Moreover, due to \eqref{11-29}, we may assume that $\az_1=1=\az_2$.
Then, by \cite[(3.3) and (3.4)]{sy15}, we find that, for all
$\lz\in(n/p_-,\fz)$ and $x\in\rn$,
$$\psi_{L,\triangledown,1}^\ast(f)(x)
\ls N_\lz^1\lf(\psi_2(t\sqrt L)\r)(f)(x),$$
which, together with \eqref{max-f1-x}, implies that \eqref{max-f2} holds true.
This finishes the proof of Lemma \ref{l-lemma2}.
\end{proof}

We now show Theorem \ref{t-max-c}.

\begin{proof}[Proof of Theorem \ref{t-max-c}]
We first prove that, for any $q\in(1,\fz]$ and $M\in(\frac n{2}[\frac1{p_-}-1],\fz)\cap\nn$,
\begin{equation}\label{11-25y}
\lf[H_{L,{\rm at},M}^{p(\cdot),q}(\rn)\cap L^2(\rn)\r]
\st \lf[H_{L,\max}^{p(\cdot),\cf}(\rn)\cap L^2(\rn)\r].
\end{equation}
Let $f\in H_{L,{\rm at},M}^{p(\cdot),q}(\rn)\cap L^2(\rn)$.
Then $f$ has a representation:
$f=\sum_{j\in\nn}\lz_j\az_j$ in $L^2(\rn)$,
where $\{\lz_j\}_{j\in\nn}\st\cc$ and $\{\az_j\}_{j\in\nn}$ is a sequence
of $(p(\cdot),q,M)_L$-atoms associated with cubes $\{Q_j\}_{j\in\nn}$ of $\rn$
such that
$\cb(\{\lz_j\az_j\}_{j\in\nn})\ls \|f\|_{H_L^{p(\cdot)}(\rn)}$.

For any $\phi\in \cf(\rr)$ and $x\in\rn$, let
$\wz\psi(x):=[\phi(0)]^{-1}\phi(x)-e^{-x^2}$.
Then, by an argument similar to that used in the proof of \cite[(3.4)]{sy15}
(see also \cite[p.\,18]{YYs15}), we
conclude that, for any $\lz\in(0,\fz)$, there exists a positive constant $C$,
depending on $n$, $\Psi$ and $\lz$, such that, for all $\phi\in\cf(\rr)$,
\begin{align*}
\sup_{|w|<t}\int_{\urn}\lf|K_{\wz\psi(t\sqrt L)\Psi(s\sqrt L)}(x-w,z)\r|
\lf[1+\frac{|x-z|}{s}\r]^\lz\,\frac{dzds}{s}\le C,
\end{align*}
where $\Psi$ is as in Lemma \ref{l-lemma1}. From this estimate and some arguments similar
to those used in the proofs of \eqref{max-f2} and \cite[(3.3) and (3.4)]{sy15},
we deduce that
\begin{equation*}
\lf\|\sup_{\phi\in\cf(\rr)}\wz\psi_{L,\triangledown,1}^{\ast}(f)\r\|_{\vlp}
\ls \|f_{L,\triangledown}^\ast\|_\vlp,
\end{equation*}
where $f_{L,\triangledown}^\ast$ is as in Definition \ref{d-max-f}.
Since
$$\cg_{L,\triangledown}^\ast(f)
\ls \sup_{\phi\in\cf(\rr)}\wz\psi_{L,\triangledown,1}^{\ast}(f)+f_{L,\triangledown}^\ast,$$
it follows that, to prove $\|\cg_{L,\triangledown}^\ast(f)\|_{\vlp}\ls\|f\|_{H_L^{p(\cdot)}(\rn)}$,
we only need to show that
\begin{equation}\label{max-c1}
\|f_{L,\triangledown}^\ast\|_{\vlp}\ls \|f\|_{H_L^{p(\cdot)}(\rn)}.
\end{equation}

To prove this, we claim that, for any $(p(\cdot),q,M)_L$-atom
$\az$ associated with some cube $Q:=Q(x_Q,\ell(Q))\st\rn$ for some $x_Q\in\rn$ and $\ell(Q)\in(0,\fz)$
and
$x\in (4\sqrt nQ)^\complement$,
\begin{equation}\label{max-c2}
\az_{L,\triangledown}^\ast(x)\ls \frac{[\ell(Q)]^{n+\delta}}{|x-x_{Q}|^{n+\delta}}
\frac1{\|\chi_{Q}\|_\vlp},
\end{equation}
where $\delta\in (n[1/p_--1],2M)$.
If this claim holds true, then,
observing that
\begin{align*}
\|f_{L,\triangledown}^\ast\|_{\vlp}
\ls \lf\|\sum_{j\in\nn}|\lz_j|(\az_j)_{L,\triangledown}^\ast
\chi_{4\sqrt nQ_j}\r\|_\vlp
+\lf\|\sum_{j\in\nn}|\lz_j|(\az_j)_{L,\triangledown}^\ast
\chi_{(4\sqrt nQ_j)^\complement}\r\|_\vlp,
\end{align*}
by Remark \ref{r-1.7x} and some argument similar to that used in the
proof of \eqref{SL-pi}, we conclude that \eqref{max-c1} holds true.

Therefore, it remains to prove the above claim. For any given
$x\in (4\sqrt nQ)^\complement$, let
$$\az_{L,\triangledown}^{\ast,1}(x):=\sup_{t\in(0,\ell(Q)),\,|y-x|<t}
\lf|e^{-t^2L}(\az)(y)\r|$$
and
$$\az_{L,\triangledown}^{\ast,2}(x):=\sup_{t\in[\ell(Q),\fz),\,|y-x|<t}
\lf|e^{-t^2L}(\az)(y)\r|.$$
Notice that, for all $t\in(0,\fz)$, $z\in Q$ and $y\in\rn$ with $|y-x|<t$, we have
\begin{equation}\label{atom-c3}
t+|y-z|> |x-z|\ge|x-x_Q|-|z-x_Q|\ge |x-x_Q|/2.
\end{equation}
Then, by Assumption \ref{a-a2} and the H\"older inequality, we find that
\begin{align}\label{atom-c2}
\az_{L,\triangledown}^{\ast,1}(x)
&\ls\sup_{t\in(0,\ell(Q)),\,|y-x|<t}\int_\rn
\frac1{t^n}e^{-\frac{|y-z|^2}{ct^2}}|\az(z)|\,dz\\
&\ls \sup_{t\in(0,\ell(Q)),\,|y-x|<t}\int_Q
\frac{t^\delta}{(t+|y-z|)^{n+\delta}}|\az(z)|\,dz\noz\\
&\ls \frac{[\ell(Q)]^\delta}{|x-x_Q|^{n+\delta}}|Q|^{1-1/q}\|\az\|_{L^q(\rn)}
\ls\frac{[\ell(Q)]^{n+\delta}}{|x-x_Q|^{n+\delta}}\frac1{\|\chi_Q\|_{\vlp}}.\noz
\end{align}
On the other hand, letting $\az:=L^Mb$ be as in Definition \ref{d-atom},
from \eqref{partial1}, \eqref{atom-c3} and
the fact that $M>\delta/2$, we deduce that
\begin{align*}
\az_{L,\triangledown}^{\ast,2}(x)
&=\sup_{t\in[\ell(Q),\fz),\,|y-x|<t}\lf|e^{-t^2L}L^Mb(y)\r|\\
&=\sup_{t\in[\ell(Q),\fz),\,|y-x|<t}t^{-2M}\lf|(t^2L)^Me^{-t^2L}(b)(y)\r|\\
&\ls \sup_{t\in[\ell(Q),\fz),\,|y-x|<t}t^{-2M}
\int_Q\frac{t^\delta}{(t+|z-y|)^{n+\delta}}|b(z)|\,dz\\
&\ls \sup_{t\in[\ell(Q),\fz),\,|y-x|<t}\frac{t^{\delta-2M}}{|x-x_Q|^{n+\delta}}
|Q|^{1-1/q}\|b\|_{L^q(\rn)}\\
&\ls \frac{[\ell(Q)]^{n+\delta}}{|x-x_Q|^{n+\delta}}\frac1{\|\chi_Q\|_{\vlp}}.
\end{align*}
By this and \eqref{atom-c2}, we conclude that, for all $x\in(4\sqrt nQ)^\complement$,
$$\az_{L,\triangledown}^\ast(x)\le \az_{L,\triangledown}^{\ast,1}(x)
+\az_{L,\triangledown}^{\ast,2}(x)\ls
\frac{[\ell(Q)]^{n+\delta}}{|x-x_Q|^{n+\delta}}\frac1{\|\chi_Q\|_{\vlp}},$$
namely, \eqref{max-c2} holds true.
This finishes the proof of \eqref{11-25y}.

Next, we show that
\begin{equation}\label{11-25z}
\lf[H_{L,\max}^{p(\cdot),\phi,a}(\rn)\cap L^2(\rn)\r]
\st \lf[H_{L,{\rm at},M}^{p(\cdot),\fz}(\rn)\cap L^2(\rn)\r].
\end{equation}
To this end, by Lemma \ref{l-lemma2}, it suffices to prove that, if
$f\in H_{L,\max}^{p(\cdot)}(\rn)\cap L^2(\rn)$, then
$f\in H_{L,{\rm at},M}^{p(\cdot),\fz}(\rn)$ and
\begin{equation}\label{11-25u}
\|f\|_{H_{L,{\rm at},M}^{p(\cdot),\fz}(\rn)}
\ls\|f\|_{H_{L,\max}^{p(\cdot)}(\rn)}.
\end{equation}

Let $\Phi$ be a function as in Lemma \ref{l-lemma1} and,
for all $x\in\rr$, $\Psi(x):=x^{2M}\Phi(x)$. Then, by the functional calculi,
we know that there exists a constant $C_{(\Psi)}$ such that
\begin{equation*}
f=C_{(\Psi)}\int_0^\fz\Psi(t\sqrt L)t^2Le^{-t^2L}(f)\,\frac{dt}t \quad {\rm in}\quad L^2(\rn).
\end{equation*}
Define a function $\eta$ by setting,
when $x\in\rr\backslash\{0\}$,
$$\eta(x):=C_{(\Psi)}\int_1^\fz t^2x^2\Psi(tx)e^{-t^2x^2}\,\frac{dt}t$$
and $\eta(0)=1$. Then $\eta\in\cs(\rr)$ is an even function and,
for any $a,\ b\in\rr$,
$$\eta(ax)-\eta(bx)=C_{(\Psi)}\int_a^b t^2x^2\Psi(tx)e^{-t^2x^2}\,\frac{dt}t,$$
which implies that
$$C_{(\Psi)}\int_a^b \Psi(t\sqrt L)t^2Le^{-t^2L}(f)\,\frac{dt}t
=\eta(a\sqrt L)(f)-\eta(b\sqrt L)(f).$$
Let, for all $x\in\rn$,
$$\cn_L^\ast(f)(x):=\sup_{t\in(0,\fz),\, |y-x|<5\sqrt nt}\lf(\lf|t^2Le^{-t^2L}f(y)\r|
+\lf|\eta(t\sqrt L)f(y)\r|\r).$$
Then, by Lemma \ref{l-lemma2}, we know that
\begin{equation}\label{11-2y}
\|\cn_L^\ast(f)\|_\vlp\ls \lf\|f_{L,\triangledown}^\ast\r\|_{\vlp}
\sim\|f\|_{H_{L,\max}^{p(\cdot)}(\rn)}.
\end{equation}

Now, following \cite[p.\,476]{sy15}, for $i\in\zz$, let $O_i:=\{x\in\rn:\ \cn_L^\ast(f)(x)\ge 2^i\}$.
Denote by $\{Q_{i,j}\}_{j\in\nn}$ the Whitney decomposition of $O_i$.
For each $i\in\zz$ and $j\in\nn$, let
$$\wh {O_i}:=\{(x,t)\in\rn:\ B(x,4\sqrt nt)\st O_i\}$$
and
$$\wz Q_{i,j}:=\{(y,t)\in\urn:\ y+3te_0\in Q_{i,j}\},$$
here and hereafter, $e_0:=(\overbrace{1,\,\dots,\,1}^{n\,{\rm times}})\in\rn$. 
Then it is easy to prove that
$\wh {O_i}\st \bigcup_{j\in\nn}\wz Q_{i,j}$ (see \cite[p.\,476]{sy15} for more details).
Observe that, for each fixed $i\in\zz$,
when $j\neq k$, $\wz Q_{i,j}\cap \wz Q_{i,k}=\emptyset$. It follows that
$$\urn =\bigcup_{i\in\zz}\wh O_i
=\bigcup_{i\in\zz}\wh {O_i}\backslash \wh{O_{i+1}}
=\bigcup_{i\in\zz}\bigcup_{j\in\nn}T_{i,j},$$
where $T_{i,j}:=\wz Q_{i,j}\cap(\wh{O_i}\backslash \wh{O_{i+1}})$. Thus,
\begin{equation*}
f=\sum_{i\in\zz}\sum_{j\in\nn}C_{(\Psi)}\int_0^\fz
\Psi(t\sqrt L)\lf(\chi_{T_{i,j}}t^2Le^{-t^2L}(f)\r)\,\frac{dt}t
=:\sum_{i\in\zz}\sum_{j\in\nn}\lz_{i,j}\az_{i,j}
\end{equation*}
converges in $L^2(\rn)$ due to the fact that $f\in L^2(\rn)$ (see \cite[(3.11)]{sy15}),
where, for any $i\in\zz$ and $j\in\nn$, $\lz_{i,j}:=2^i\|\chi_{Q_{i,j}}\|_{\vlp}$
and $a_{i,j}:=L^M(b_{i,j})$ with
$$b_{i,j}:=\frac{C_{(\Psi)}}{\lz_{i,j}}\int_0^\fz\Psi(t\sqrt L)
\lf(\chi_{T_{i,j}}t^2Le^{-t^2L}(f)\r)\,\frac{dt}t.$$
By an argument similar to that used in \cite[pp.\,477-479]{sy15}, we find that there
exists a positive constant $\wz C$ such that,
for each $i\in\zz$ and $j\in\nn$,
$\wz Ca_{i,j}$ is a $(p(\cdot),\fz,M)_L$-atom associated with the cube $30Q_{i,j}$.
Moreover, by Lemma \ref{l-bigsball}, Remark \ref{r-hlmo} and \eqref{11-2y}, we conclude that
\begin{align*}
&\ca\lf(\{\lz_{i,j}\}_{j\in\nn},\{30Q_{i,j}\}_{j\in\nn}\r)\\
&\hs\ls\lf\|\lf\{\sum_{i\in\zz,\,j\in\nn}\lf[\frac{\lz_{i,j}\chi_{Q_{i,j}}}
{\|\chi_{Q_{i,j}}\|_\vlp}\r]^{p_-}\r\}^{\frac 1{p_-}}\r\|_\vlp
\ls \lf\|\lf\{\sum_{i\in\zz,\,j\in\nn}\lf[2^i\chi_{Q_{i,j}}\r]^{p_-}
\r\}^{\frac 1{p_-}}\r\|_{\vlp}\\
&\hs\sim \lf\|\lf\{\sum_{i\in\zz}\lf[2^i\chi_{O_{i}}\r]^{p_-}
\r\}^{\frac 1{p_-}}\r\|_{\vlp}
\sim \lf\|\lf\{\sum_{i\in\zz}\lf[2^i\chi_{O_{i}\backslash O_{i+1}}
\r]^{p_-}\r\}^{\frac 1{p_-}}\r\|_{\vlp}\\
&\hs\sim\|\cn_L^\ast(f)\|_\vlp
\ls\|f\|_{H_{L,\max}^{p(\cdot)}(\rn)}<\fz,
\end{align*}
which implies that
$f\in H_{L,\rm{at},M}^{p(\cdot),\fz}(\rn)$ and hence \eqref{11-25u} holds true.
This finishes the proof of \eqref{11-25z}.

Finally, by Lemma \ref{l-lemma2} and the definitions of
$H_{L,\max}^{p(\cdot),\phi,a}(\rn)$ and $H_{L,\max}^{p(\cdot),\cf}(\rn)$,
we immediately find that
\begin{equation*}
\lf[H_{L,\max}^{p(\cdot),\cf}(\rn)\cap L^2(\rn)\r]
\st \lf[H_{L,\max}^{p(\cdot),\phi,a}(\rn)\cap L^2(\rn)\r],
\end{equation*}
which, together with \eqref{11-25y}, \eqref{11-25z} and Remark \ref{r-atom}, implies
that, for any $q\in(1,\fz]$ and $M\in(\frac n{2}[\frac1{p_-}-1],\fz)\cap\nn$,
$$\lf[H_{L,\rm{at},M}^{p(\cdot),q}(\rn)\cap L^2(\rn)\r]
=\lf[H_{L,\max}^{p(\cdot),\phi,a}(\rn)\cap L^2(\rn)\r]
=\lf[H_{L,\max}^{p(\cdot),\cf}(\rn)\cap L^2(\rn)\r].$$
From this, Remark \ref{r-dense} and a density argument, we further deduce that the spaces
$H_{L,\rm{at},M}^{p(\cdot),q}(\rn)$, $H_{L,\max}^{p(\cdot),\phi,a}(\rn)$
and $H_{L,\max}^{p(\cdot),\cf}(\rn)$ coincide with equivalent quasi-norms,
which completes the proof of Theorem \ref{t-max-c}.
\end{proof}

\section{Proof of Theorem \ref{t-max-rad}\label{s-2.3}}
\hskip\parindent
In this section, we give the proof of Theorem \ref{t-max-rad}, via beginning with
establishing the following conclusion.

\begin{proposition}\label{p-max}
Let $L$ be as in Theorem \ref{t-max-rad} and $p(\cdot)\in C^{\log}(\rn)$.
Then there exists a positive constant $C$ such that, for all $f\in L^2(\rn)$,
\begin{equation}\label{max-f2z}
\lf\|f_{L,\triangledown}^\ast\r\|_{\vlp}
\le C\lf\|f_{L,+}^\ast\r\|_{\vlp}.
\end{equation}
\end{proposition}

To prove Proposition \ref{p-max}, we need several auxiliary estimates.

\begin{lemma}\label{l-max2}
Let $p(\cdot)\in C^{\log}(\rn)$ and $\lz\in(n/p_-,\fz)$.
Then there exists a positive constant $C$ such that, for any measurable function
$F$ on $\urn$,
\begin{equation}\label{max-f2x}
\lf\|N_\lz^1(F)\r\|_\vlp\le C\|M_1(F)\|_{\vlp},
\end{equation}
where $N_\lz^1(F)$ and $M_1(F)$ are as in \eqref{11-26y}, respectively, \eqref{11-26z}.
\end{lemma}

\begin{proof}
To prove this lemma, it suffices to show that, for all $x\in\rn$,
\begin{equation}\label{max-f2y}
N_\lz^1(F)(x)\le \lf\{\cm\lf(\lf[M_1(F)\r]^{n/\lz}\r)(x)
\r\}^{\lz/n},
\end{equation}
where $\cm$ denotes the Hardy-Littlewood maximal function defined in Remark \ref{r-1.7x}.
Indeed, if \eqref{max-f2y} is proved, then, by Remark \ref{r-vlp-m} and the fact that
$\lz\in(n/p_-,\fz)$, we find that \eqref{max-f2x} holds true.

Next we show \eqref{max-f2y}. By the definition of $M_1(F)$,
we know that, for any $t\in(0,\fz)$, $x,\ y\in\rn$ and $z\in B(x-y,t)$,
$\lf|F(x-y,t)\r|\le M_1(F)(z)$.
From this and the fact that $B(x-y,t)\st B(x,|y|+t)$,
we deduce that
\begin{align*}
\lf|F(x-y,t)\r|^{n/\lz}
&\le\frac1{|B(x-y,t)|}\int_{B(x,|y|+t)}\lf[M_1(F)(z)
\r]^{n/\lz}\,dz\\
&\le \lf(1+\frac{|y|}t\r)^n\cm\lf(\lf[M_1(F)
\r]^{n/\lz}\r)(x),
\end{align*}
which further implies that \eqref{max-f2y} holds true. This finishes the proof of
Lemma \ref{l-max2}.
\end{proof}

For any $\ez,\ N\in(0,\fz)$, $f\in L^2(\rn)$ and $x\in\rn$, let
\begin{equation*}
f_{L,+}^{\ast,\ez,N}(x):=\sup_{t\in(0,\fz)}\lf|e^{-t^2L}(f)(x)\r|
\frac{t^N}{[(t+\ez)(1+\ez|x|)]^N},
\end{equation*}
\begin{equation}\label{11-26x}
f_{L,\triangledown}^{\ast,\ez,N}(x):=\sup_{t\in(0,1/\ez),\,|x-y|<t}
\lf|e^{-t^2L}(f)(y)\r|
\frac{t^N}{[(t+\ez)(1+\ez|y|)]^N}
\end{equation}
and, for all $\lz\in(0,\fz)$,
\begin{equation*}
M_{L}^{\lz,\ez,N}(f)(x):=\sup_{t\in(0,1/\ez),\,y\in\rn}
\lf|e^{-t^2L}(f)(y)\r|\lf(1+\frac{|x-y|}t\r)^{-\lz}\frac{t^N}{[(t+\ez)(1+\ez|y|)]^N}.
\end{equation*}

By an argument similar to that used in the proof of \eqref{max-f2y}, we obtain the
following conclusion, the details being omitted.

\begin{lemma}\label{l-max3}
Let $L$ be as in Theorem \ref{t-max-rad} and $p(\cdot)\in C^{\log}(\rn)$.
Suppose that $\lz\in(0,\fz)$ and $\phi\in\cs(\rr)$ is an
even function with $\phi(0)=1$. Then it holds true that,
for all $\ez,\ N\in(0,\fz)$, $f\in L^2(\rn)$ and $x\in\rn$,
\begin{equation*}
M_{L}^{\lz,\ez,N}(f)(x)\le
\lf\{\cm\lf(\lf[f_{L,\triangledown}^{\ast,\ez,N}\r]^{n/\lz}\r)(x)
\r\}^{\lz/n}.
\end{equation*}
\end{lemma}

Moreover, we have the following lemma.

\begin{lemma}\label{l-max4}
Let $L$ be as in Theorem \ref{t-max-rad} and $p(\cdot)\in C^{\log}(\rn)$.
For any $\gz,\ \lz,\ \ez,\ N\in(0,\fz)$ and $f\in L^2(\rn)$, let
$$E:=\lf\{x\in\rn:\ M_{L}^{\lz,\ez,N}(f)(x)
\le\gz f_{L,\triangledown}^{\ast,\ez,N}(x)\r\}.$$
Then there exists a positive constant $C$, independent of $\ez,\ N$ and $f$,
such that, for all $x\in E$,
\begin{equation}\label{max-f4}
f_{L,\triangledown}^{\ast,\ez,N}(x)
\le C\lf\{\cm\lf(\lf[f_{L,+}^\ast\r]^{n/\lz}\r)(x)\r\}^{\lz/n}.
\end{equation}
\end{lemma}

\begin{proof}
Let $x$ be a given point of $E\st \rn$. Then, by the definition of $f_{L,\triangledown}^{\ast,\ez,N}(x)$,
we easily know that there exists $(y_0,t_0)\in\urn$ such that $t_0\in(0,1/\ez)$,
$|x-y_0|<t_0$ and
\begin{equation}\label{max-f4x}
f_{L,\triangledown}^{\ast,\ez,N}(x)\le 2\lf|e^{-t_0^2L}(f)(y_0)\r|
\frac{t_0^N}{[(t_0+\ez)(1+\ez|y_0|)]^N}.
\end{equation}

We claim that, for any $s\in(0,1)$, $r\in(0,\fz)$ and $\wz x\in B(y_0,rt_0)$,
\begin{align}\label{max-f4y}
{\rm I}(x,\wz x,y_0,r,t_0)
:&=\lf|e^{-t_0^2L}(f)(\wz x)-e^{-t_0^2L}(f)(y_0)\r|\\
&\ls r^{\mu s}M_{L}^{\lz,\ez,N}(f)(x)\lf[\frac{t_0}{(t_0+\ez)(1+\ez|y_0|)}\r]^{-N},\noz
\end{align}
where $\mu$ is as in Assumption \ref{a-a3}.
If this claim holds true, then,
by choosing $r$ small enough, we find that, for any $\wz x\in B(y_0,rt_0)$,
\begin{align*}
{\rm I}(x,\wz x,y_0,r,t_0)
&\ls r^{\mu s}\gz f_{L,\triangledown}^{\ast,\ez,N}(x)
\lf[\frac{t_0}{(t_0+\ez)(1+\ez|y_0|)}\r]^{-N}\\
&\le \frac14f_{L,\triangledown}^{\ast,\ez,N}(x)
\lf[\frac{t_0}{(t_0+\ez)(1+\ez|y_0|)}\r]^{-N},
\end{align*}
which, combined with \eqref{max-f4x}, implies that, for any $\wz x\in B(y_0,rt_0)$,
\begin{equation*}
\lf|e^{-t_0^2L}(f)(\wz x)\r|
\ge \frac14f_{L,\triangledown}^{\ast,\ez,N}(x)
\lf[\frac{t_0}{(t_0+\ez)(1+\ez|y_0|)}\r]^{-N}
\ge \frac14f_{L,\triangledown}^{\ast,\ez,N}(x).
\end{equation*}
Therefore, for all $x\in\rn$, we have
\begin{align*}
\lf[f_{L,\triangledown}^{\ast,\ez,N}(x)\r]^{n/\lz}
&\ls \frac1{|B(y_0,rt)|}\int_{B(y_0,rt_0)}\lf|e^{-t_0^2L}(f)(\wz x)\r|^{n/\lz}\,d\wz x\\
&\ls \lf(\frac{1+r}{r}\r)^n\frac1{|B(x,(1+r)t_0)|}
\int_{B(x,(1+r)t_0)}\lf|e^{-t_0^2L}(f)(\wz x)\r|^{n/\lz}\,d\wz x\\
&\ls \cm\lf(\lf[f_{L,+}^\ast\r]^{n/\lz}\r)(x)
\end{align*}
with $\cm$ as in Remark \ref{r-1.7x}, namely, \eqref{max-f4} holds true.

To complete the proof of Lemma \ref{l-max4}, it remains to show \eqref{max-f4y}.
By the semigroup property of $\{e^{-tL}\}_{t>0}$, we know that
\begin{align}\label{max-f4u}
{\rm I}(x,\wz x,y_0,r,t_0)
&=\lf|\int_\rn\lf[K_{{t_0^2}/2}(\wz x,z)-K_{{t_0^2}/2}(y_0,z)\r]
e^{-{t_0^2}L/2}(f)(z)\,dz\r|\\
&\le {\rm I}_0+\sum_{k=3}^\fz{\rm I}_k,\noz
\end{align}
where
$${\rm I}_0:=\int_{B(y_0,4t_0)}
\lf|K_{{t_0^2}/2}(\wz x,z)-K_{{t_0^2}/2}(y_0,z)\r|
\lf|e^{-{t_0^2}L/2}(f)(z)\r|\,dz$$
and, for each $k\in\{3,4,\dots\}$,
$${\rm I}_k:=\int_{U_k(y_0,t_0)}
\lf|K_{{t_0^2}/2}(\wz x,z)-K_{{t_0^2}/2}(y_0,z)\r|
\lf|e^{-{t_0^2}L/2}(f)(z)\r|\,dz$$
with $U_k(y_0,t_0):=B( y_0,2^kt_0)\backslash B(y_0,2^{k-1}t_0)$.

Observe that, for any $s\in(0,1)$, by Assumptions \ref{a-a1} and \ref{a-a2}, we find that,
for all $z\in\rn$,
\begin{align}\label{max-f4z}
\lf|K_{{t_0^2}/2}(\wz x,z)-K_{{t_0^2}/2}(y_0,z)\r|
\ls\frac1{t_0^n}\lf[\frac{|\wz x-y_0|}{t_0}\r]^{\mu s}
\lf[e^{-\frac{|\wz x-z|^2}{ct_0^2}}+e^{-\frac{|y_0-z|^2}{ct_0^2}}\r]^{1-s},
\end{align}
where $c$ and $\mu$ are as in Assumptions \ref{a-a2}, respectively, \ref{a-a3}.
By this and the fact that, for any $z\in B(y_0,4t_0)$,
$|x-z|\le5t_0$ and
$$\frac{1+\ez|z|}{1+\ez|y_0|}\le \frac{1+\ez|z-y_0|+\ez|y_0|}{1+\ez|y_0|}\ls 1,$$
we find that
\begin{align}\label{max-f4v}
{\rm I}_0
&\ls \int_{B(y_0,4t_0)}\frac1{t_0^n}\lf(\frac{|\wz x-y_0|}{t_0}\r)^{\mu s}
e^{-t_0^2L/2}(f)(z)\,dz\\
&\ls \frac{r^{\mu s}}{t_0^n}M_{L}^{\lz,\ez,N}(f)(x)
\int_{B(y_0,4t_0)}\lf(1+\frac{|x-z|}{t_0}\r)^\lz
\frac{[(t_0+\ez)(1+\ez|z|)]^N}{t_0^N}\,dz\noz\\
&\ls r^{\mu s}M_{L}^{\lz,\ez,N}(f)(x)
\frac{[(t_0+\ez)(1+\ez|y_0|)]^N}{t_0^N}.\noz
\end{align}

Next we deal with I$_k$ for all $k\in\{3,4,\dots\}$. Since $|\wz x-y_0|<t_0$,
it follows that, for any
$z\in U_k(y_0,t_0)$, $|\wz x-y_0|\le |y_0-z|/4$ and hence
$$|\wz x-z|\ge|z-y_0|-|y_0-\wz x|>|y_0-z|/2.$$
From this, \eqref{max-f4z} and the fact that,
for any $z\in U_k(y_0,t_0)$,
$$\frac{1+\ez|z|}{1+\ez|y_0|}\le 1+\frac{\ez|z-y_0|}{1+\ez|y_0|}\ls 2^k,$$
we deduce that
\begin{align}\label{max-f4w}
{\rm I}_k
&\ls \frac{r^{\mu s}}{t_0^n}\int_{U_k(y_0,t_0)}
e^{-\frac{|y_0-z|^2}{2ct_0^2}(1-s)}\lf|e^{-t_0^2L/2}(f)(z)\r|\,dz\\
&\ls \frac{r^{\mu s}}{t_0^n}M_L^{\lz,\ez,N}(f)(x)2^{\lz k}e^{-\beta2^{2k}}
\int_{U_k(y_0,t_0)}
\frac{[(t_0+\ez)(1+\ez|z|)]^N}{t_0^N}\,dz\noz\\
&\ls r^{\mu s}2^{k(\lz+n+N)}e^{-\beta2^{2k}}M_L^{\lz,\ez,N}(f)(x),\noz
\end{align}
where $\beta$ is a positive constant depending on $s$ and $c$.

Combining \eqref{max-f4u}, \eqref{max-f4v} and \eqref{max-f4w}, we conclude that
\eqref{max-f4y} holds true.
This finishes the proof of Lemma \ref{l-max4}.
\end{proof}

We also need the following technical lemma.

\begin{lemma}\label{l-max5}
Let $\ez\in(0,1)$ and $f\in L^2(\rn)$.

{\rm(i)} It holds true that there exists a positive constant $N$, depending on $f$, such that
$f_{L,\triangledown}^{\ast,\ez,N}\in\vlp$.

{\rm (ii)} If $f_{L,\triangledown}^\ast\in\vlp$, then
$f_{L,\triangledown}^\ast\in \vlp$.
\end{lemma}

\begin{proof}
We first prove (i).
Let $\vz(x):=e^{-x^2/c}$ for all $x\in\rn$, where $c$ is as in Assumption \ref{a-a2}.
Then $\vz\in\cs(\rn)$ and, by Assumption \ref{a-a2}, we know that, for all $y\in\rn$,
\begin{equation*}
\lf|e^{-t^2L}(f)(y)\r|
\ls \int_{\rn}\frac1{t^n}e^{-\frac{|y-z|^2}{ct^2}}|f(z)|\,dz
\sim \vz_t\ast(|f|)(y),
\end{equation*}
where, for $t\in(0,\fz)$, $\vz_t(\cdot):=t^{-n}\vz(\frac{\cdot}{t})$,
which, combined with \cite[Theorem 2.3.20]{Gra1}, implies that there exist a
positive constant $C_{(f)}$ and integers $m$ and $l$, depending on $f$, such that,
for all $y\in\rn$,
\begin{align*}
\lf|e^{-t^2L}(f)(y)\r|
&\le C_{(f)}\sum_{\beta\in\zz_+^n,|\beta|\le l}\sup_{z\in\rn}(|y|^m+|z|^m)
|(\partial^\beta\vz_t)(z)|\\
&\le C_{(f)}\frac{(1+|y|)^m}{\min\{t^n,t^{n+l}\}}(1+t^m)
\sum_{\beta\in\zz_+^n,|\beta|\le l}
\sup_{z\in\rn}(1+|z/t|^m)|(\partial^\beta\vz)(z/t)|\\
&\le C_{(f)}(1+\ez|y|)^m\ez^{-m}(1+t^m)(t^{-n}+t^{-n-l}).
\end{align*}
From this, we further deduce that, for all $t\in(0,1/\ez)$ and $|y-x|<t$,
\begin{align*}
\lf|e^{-t^2L}(f)(y)\r|\frac{t^N}{[(t+\ez)(1+\ez|y|)]^N}
\le C_{(f)} \frac1{(1+\ez|y|)^{N-m}}\frac{1+\ez^{-m}}{\ez^{m+N/2}}
\lf(\ez^{n-N/2}+\ez^{n+l-N/2}\r),
\end{align*}
where $N$ is chosen large enough such that $N>\max\{2(n+l),m+n/p_-\}$,
which, together with the fact that $\ez\in(0,1)$ and
$1+\ez|y|\ge\frac 12(1+\ez|x|)$, implies that
\begin{equation*}
f_{L,\triangledown}^{\ast,\ez,N}(x)
\le C_{(f)}\frac1{(1+\ez|x|)^{N-m}}\frac{1+\ez^{-m}}{\ez^{m+N-n}}.
\end{equation*}
Observe that, for all $x\in\rn$,
\begin{equation*}
(1+\ez|x|)^{m-N}\le \ez^{m-N}(1+|x|)^{m-N}
\ls \ez^{m-N}\lf[\cm(\chi_{B(0,1)})(x)\r]^{(N-m)/n},
\end{equation*}
where $\cm$ denotes the Hardy-Littlewood maximal function defined in Remark \ref{r-1.7x}.
By this and Remark \ref{r-vlp-m}, we conclude that
\begin{align*}
\lf\|f_{L,\triangledown}^{\ast,\ez,N}\r\|_{\vlp}
\ls \lf\|\cm(\chi_{B(0,1)})\r\|_{L^{\frac{N-m}np(\cdot)}}^{(N-m)/n}
\ls \|\chi_{B(0,1)}\|_{\vlp}<\fz,
\end{align*}
where the implicit positive constants depend on $f,\ n,\ N$ and $\ez$.
Therefore, $f_{L,\triangledown}^{\ast,\ez,N}\in\vlp$.

Next, we show (ii).
For any $\lz\in(n/p_-,\fz)$ and $\gz\in(0,\fz)$,
let $E$ be as in Lemma \ref{l-max4}.
Then, by Lemma \ref{l-max3} and Remark \ref{r-vlp-m}, we conclude that
\begin{equation*}
\lf\|f_{L,\triangledown}^{\ast,\ez,N}\chi_{E^\complement}\r\|_\vlp
\le \frac1\gz\lf\|M_{L}^{\lz,\ez,N}(f)\r\|_\vlp
\le \frac{C_1}\gz\lf\|f_{L,\triangledown}^{\ast,\ez,N}\r\|_\vlp,
\end{equation*}
which, combined with Remark \ref{r-vlp}, implies that
\begin{align*}
\lf\|f_{L,\triangledown}^{\ast,\ez,N}\r\|_\vlp^{p_-}
&\le \lf\|f_{L,\triangledown}^{\ast,\ez,N}\chi_E\r\|_\vlp^{p_-}
+\lf(\frac{C_1}{\gz}\r)^{p_-}\lf\|f_{L,\triangledown}^{\ast,\ez,N}\r\|_\vlp^{p_-}
\end{align*}
with the positive constant $C_1$ independent of $f$.
By this, Lemma \ref{l-max5} and choosing $\gz:=2^{1/p_-}C_1$, we find that
\begin{equation*}
\lf\|f_{L,\triangledown}^{\ast,\ez,N}\r\|_\vlp
\le 2^{1/p_-}\lf\|f_{L,\triangledown}^{\ast,\ez,N}\chi_E\r\|_\vlp.
\end{equation*}
From this, Lemma \ref{l-max4} and Remark \ref{r-vlp-m}, we deduce that
\begin{align}\label{11-24x}
\lf\|f_{L,\triangledown}^{\ast,\ez,N}\r\|_\vlp
\ls \lf\|\lf\{\cm\lf([f_{L,+}^\ast]^{n/\lz}\r)\r\}^{\lz/n}\r\|_\vlp
\ls \|f_{L,+}^\ast\|_\vlp
\end{align}
with the implicit positive constants independent of $\ez$.
Notice that, for any $x\in\rn$,
\begin{equation}\label{11-24}
f_{L,\triangledown}^{\ast,\ez,N}(x)\ge \frac{2^{-N}}{(1+\ez|x|)^N}
\sup_{t\in(0,1/\ez)}\lf(\frac t{t+\ez}\r)^N
\sup_{|y-x|<t}\lf|e^{-t^2L}(f)(y)\r|
\end{equation}
and that the right hand side of \eqref{11-24} increases to
$2^{-N}f_{L,,\triangledown}^\ast(x)$ as $\ez\to0^+$, namely, $\ez\in(0,\fz)$ and
$\ez\to0$. Thus, it follows, from the Fatou lemma (see \cite[Theorem 2.61]{cfbook})
and \eqref{11-24x}, that
\begin{align}\label{11-24y}
\lf\|f_{L,\triangledown}^\ast\r\|_{\vlp}
\le 2^N\liminf_{\ez\to0^+}\|f_{L,\triangledown}^{\ast,\ez,N}\|_\vlp
\ls \|f_{L,+}^\ast\|_\vlp,
\end{align}
which implies that $f_{L,\triangledown}^{\ast}\in\vlp$ and hence
completes the proof of Lemma \ref{l-max5}.
\end{proof}

\begin{remark}
Due to \eqref{11-24y}, Proposition \ref{p-max} seems to be proved.
However, this is not the case, since the implicit
positive constant in \eqref{11-24y} depends on $N$ and hence on $f$, which is not allowed in
Proposition \ref{p-max}.
\end{remark}

Indeed, we prove Proposition \ref{p-max} by an argument similar to
that used in the proof \eqref{11-24y} and the observation that, if
$\|f_{L,+}^\ast\|_{\vlp}$ is finite, then $\|f_{L,\triangledown}^\ast\|_\vlp$
is also finite.

For an even function $\phi\in\cs(\rr)$ with $\phi(0)=1$, let, for any $\lz\in(0,\fz)$
and $x\in\rn$,
$$M_{L,\phi}^\lz(f)(x):=
N_\lz^1(\phi(t\sqrt L)(f))=\sup_{t\in(0,\fz),\, y\in\rn}\lf|\phi(t\sqrt L)(f)(y)\r|
\lf(1+\frac{|x-y|}t\r)^{-\lz}.$$
Particularly, when $\phi:=e^{-|\cdot|^2}$, we denote $M_{L,\phi}^\lz(f)$
simply by $M_L^\lz(f)$.

\begin{proof}[Proof of Proposition \ref{p-max}]
For any $\lz\in(n/p_-,\fz)$ and $\gz\in(0,\fz)$, let
$$F:=\lf\{x\in\rn:\ M_L^\lz(f)(x)\le \gz f_{L,\triangledown}^\ast(x)\r\}.$$
Then, by Lemma \ref{l-max2}, we find that
\begin{align*}
\lf\|f_{L,\triangledown}^\ast\chi_{F^\com}\r\|_\vlp
\le \frac1\gz\|M_L^\lz(f)\|_\vlp
\le \frac{C_2}\gz\lf\|f_{L,\triangledown}^\ast\r\|_\vlp,
\end{align*}
where $C_2$ is a positive constant independent of $f$.
Notice that
$$\|f_{L,\triangledown}^\ast\|_\vlp^{p_-}
\le \|f_{L,\triangledown}^\ast\chi_F\|_\vlp^{p_-}
+\|f_{L,\triangledown}^\ast\chi_{F^\com}\|_\vlp^{p_-}.$$
From this and Lemma \ref{l-max5}(ii), together with choosing $\gz:=2^{1/p_-}C_2$,
we deduce that
\begin{equation}\label{11-24z}
\|f_{L,\triangledown}^\ast\|_\vlp
\le 2^{1/p_-}\|f_{L,\triangledown}^\ast\chi_F\|_\vlp.
\end{equation}
On the other hand, by an argument similar to that used in the proof of
Lemma \ref{l-max4}, we conclude that, for all $x\in F$,
\begin{equation*}
f_{L,\triangledown}^\ast(x)\ls \lf\{\cm\lf([f_{L,+}^\ast]^{n/\lz}\r)(x)\r\}^{\lz/n}
\end{equation*}
with $\cm$ as in Remark \ref{r-1.7x}, which, combined with \eqref{11-24z} and Remark \ref{r-vlp-m}, implies that
\eqref{max-f2z} holds true. This finishes the proof of Proposition \ref{p-max}.
\end{proof}

We end this section by proving Theorem \ref{t-max-rad}.

\begin{proof}[Proof of Theorem \ref{t-max-rad}]
To show Theorem \ref{t-max-rad}, by Remark \ref{r-dense} and the definitions of
$H_{L,\max}^{p(\cdot)}(\rn)$ and $H_{L,\rad}^{p(\cdot)}(\rn)$,
we only need to prove that
\begin{equation}\label{11-25v}
\lf[H_{L,\rad}^{p(\cdot)}(\rn)\cap L^2(\rn)\r]
\st \lf[H_{L,\max}^{p(\cdot)}(\rn)\cap L^2(\rn)\r],
\end{equation}
since the inverse inclusion is obvious.

Let $f\in H_{L,\rad}^{p(\cdot)}(\rn)\cap L^2(\rn)$.
Then, by Proposition \ref{p-max}, we find that
\begin{equation*}
\|f\|_{H_{L,\max}^{p(\cdot)}(\rn)}=
\|f_{L,\triangledown}^\ast\|_\vlp \ls\|f_{L,+}^\ast\|_\vlp
\sim \|f\|_{H_{L,\rad}^{p(\cdot)}(\rn)}<\fz,
\end{equation*}
which implies that $f\in H_{L,\max}^{p(\cdot)}(\rn)\cap L^2(\rn)$
and hence \eqref{11-25v} holds true.
This finishes the proof of Theorem \ref{t-max-rad}.
\end{proof}

\bigskip

\noindent  Ciqiang Zhuo and Dachun Yang

\medskip

\noindent  School of Mathematical Sciences, Beijing Normal University,
Laboratory of Mathematics and Complex Systems, Ministry of
Education, Beijing 100875, People's Republic of China

\smallskip

\noindent {\it E-mails}: \texttt{cqzhuo@mail.bnu.edu.cn} (C. Zhuo)

\hspace{0.98cm} \texttt{dcyang@bnu.edu.cn} (D. Yang)

\end{document}